\documentclass[reqno]{amsart}
\usepackage{a4wide,calrsfs}
\usepackage{amsmath,bbm}
\usepackage{color}
\usepackage{epsfig}
\usepackage{subfigure}
\usepackage{amssymb}
\usepackage{enumerate}
\usepackage{calrsfs,amssymb,undertilde}
\usepackage{amsthm}
\usepackage{amsopn}
\usepackage{upgreek}

\newcommand{\bpsi}{\boldsymbol{\psi}}

\newcommand{\bvarphi}{\boldsymbol{\varphi}}
\newcommand{\R}{\mathbb{R}}

\newcommand{\bu}{\mathbf{u}}
\newcommand{\bw}{\mathbf{w}}
\newcommand{\bF}{\mathbf{F}}
\newcommand{\bG}{\mathbf{G}}
\newcommand{\bH}{\mathbf{H}}
\newcommand{\bL}{\mathbf{L}}
\newcommand{\bV}{\mathbf{V}}
\newcommand{\bW}{\mathbf{W}}
\newcommand{\bx}{\mathbf{x}}

\newcommand{\bm}{\mathbf{m}}
\newcommand{\bn}{\mathbf{n}}
\newcommand{\bq}{\mathbf{q}}
\newcommand{\bv}{\mathbf{v}}

\newcommand{\diver}{\mathrm{div}\,}
\newcommand{\curl}{\mathrm{curl}\,}
\DeclareMathOperator*{\esssup}{ess\,sup}

\newtheorem{definition}{Definition}[section]
\newtheorem{theorem}{Theorem}[section]

\newtheorem{lemma}{Lemma}[section]

\allowdisplaybreaks
\numberwithin{equation}{section}

\title[An inverse  problem for nonhomogeneous asymmetric fluids]
{A local uniqueness result for an inverse problem 
to the system modelling nonhomogeneous asymmetric fluids}

\author[Coronel]{An{\'\i}bal\ Coronel$^\dag$}
\author[Rojas-Medar]{Marko Rojas-Medar$^\dag$}

\thanks{$^\dag$ GMA, Departamento de Ciencias B\'asicas,
Facultad de Ciencias, Universidad del B\'{\i}o-B\'{\i}o, 
Campus Fernando May, Chill\'{a}n, Chile, 
E-mail: {\tt acoronel@ubiobio.cl}, {\tt marko.medar@gmail.com}}

\date{\today}

\begin{document}

\begin{abstract}
In this paper, we prove the local uniqueness of an
inverse problem arising in the nonstationary
flow of a nonhomogeneous incompressible asymmetric fluid
in a bounded domain with smooth boundary. 
The direct problem is an
initial-boundary value problem for a system for the velocity
field, the angular velocity of rotation of the fluid particles,
the mass density and the pressure distribution. The inverse
problem consists in the external force recover 
assuming an integral  measurements on the boundary.
We  characterize the inverse problem solutions using 
an operator equation of a second kind,
which is deduced form the application of the Helmholtz decomposition.
We introduce  several estimates which implies the hypothesis
of the Tikhonov fixed point theorem.
\end{abstract}



\maketitle

\section{Introduction}

The fluids with density-dependent and non-symmetric behavior of the stress tensor 
belong to a widely class of fluids which are relevant in 
many industrial applications and in 
several areas of science. Concerning to the industrial and
laboratory applications, 
this kind of fluids appears for instance
in animal blood flow, lubrication theory,
polymer suspensions and liquid crystals 
\cite{kim_2005,lukaszewicz_book,petrosyan_book}. 
Now, among the sciences
we have the following: physics, partial differential equations,
functional analysis, control theory, numerical analysis, biology, hydrodynamics, etc 
\cite{aero_1965,antontsev_1990,ashraf_2009,ladyzhenskaya_book}.
It is known that there exists several theories to describe the 
behavior of this kind of fluids but there is  not still  
a universal theory to describe all of them. In particular,
one of the most important approaches has been done by
A. C. Eringen in \cite{eringen_1964} (see also \cite{eringen_1966}), 
where two relevant facts are introduced. First,
it is introduced the concept of micropolar fluids 
to characterize the fluids consisting of rigid,
randomly oriented (or spherical) particles suspended in a viscous medium,
where the deformation of fluid particles is ignored. Second,
it is deduced the mathematical model for micropolar fluids which consists of
a simple and consistent 
generalization of the classical Navier-Stokes model. 
Nowadays, the micropolar fluids are also called asymmetric fluids 
and the mathematical properties are largely studied by several authors, see 
for instance \cite{antontsev_1990,boldrini_1996,boldrini_2003,braz_2014,braz_2007,
conca_2002,lukaszewicz_1990,lukaszewicz_1989,rojasmedar_2005,vitoriano_2013}.
However, since the nonhomogeneous asymmetric model is a system 
related with Navier-Stokes equations, 
where several  open questions exists.
In particular, in this paper we
address an answer to the question of well posedness for a source inverse problem.
Moreover, for details of some other still unsolved  problems 
related with control and a geometric  inverse
problems consult the recent survey given by E. Fern\'andez-Cara 
and collaborators in~\cite{fernandez_2013}.

In this paper we study the inverse problem of determining 
the density functions $\bF$ and $\bG$, modelling the vector
external sources for the linear and the angular momentum of particles,
in a system for the motion in a finite time $t\in [0,T]$ 
of nonhomogeneous viscous incompressible asymmetric fluid 
on a bounded and regular domain $\Omega\subset\mathbb{R}^3$, 
with boundary~$\partial\Omega$:
\begin{eqnarray}
&&(\rho \bu)_t+\diver(\rho \bu\otimes \bu)
	-(\mu +\mu_r)\Delta \bu+\nabla p
	=2\mu_r \curl \bw+\rho \bF,
	\quad\mbox{in}\quad Q_T:=\Omega\times [0,T],
\label{eq:momento_lineal}
\\
&&\diver (\bu)=0,
	\quad\mbox{in}\quad Q_T,
\label{eq:incompresibilidad}
\\
&&(\rho \bw)_t+\diver(\rho \bw\otimes \bw)-(c_a+c_d)\Delta \bw
	-(c_0+c_d-c_a)\nabla\diver \bw+
	4\mu_r \bw
\nonumber
\\
&&
\hspace{2cm}
=2\mu_r \curl \bu+\rho \bG,
	\quad\mbox{in}\quad Q_T,
	\quad
	c_0+c_d>c_a,
\label{eq:momento_angular}
\\
&&
 \rho_t+\bu\cdot \nabla \rho=0,
	\quad\mbox{in}\quad Q_T,
\label{eq:ecuacion_continuidad}
\\
&&
 \bu(\bx,0)=\bu_0(\bx),\quad  \bw(\bx,0)=\bw_0(\bx),\quad  \rho(\bx,0)=\rho_0(\bx),
	\quad\mbox{on}\quad \Omega,
\label{eq:direct_problem_ic}
\\
&&
 \bu(\bx,t)=\bw(\bx,t)=0,
	\quad\mbox{on}\quad \Sigma_T:=\partial\Omega\times [0,T],
\label{eq:direct_problem_bc}
\\
&&
\int_{\Omega}\rho(\bx,t)\bu(\bx,t)\cdot\bpsi^{\bu}(\bx)d\bx=\phi^{\bu}(t),
	\quad
	\int_{\Omega}\rho(\bx,t)\bw(\bx,t)\cdot\bpsi^{\bw}(\bx)d\bx=\phi^{\bw}(t),
	\quad
	 t\in [0,T].\hspace{1cm}
\label{eq:direct_problem_ovc}
\end{eqnarray}
Here $\bu,\bw,\rho$ and $p$ denote the velocity field, the angular velocity of rotation of 
the fluid particles, the mass density and the pressure distribution, respectively. 
The constant $\mu>0$ is the usual Newtonian viscosity
and the positive  constants $\mu_r,c_0$ and $c_d$ are  the additional viscosities related
to the lack of symmetry of the stress tensor.
The functions $\bpsi^{\bu},$ $\bpsi^{\bw},$ $\phi^{\bu}$ and $\phi^{\bw}$ in 
the integral overdetermination condition \eqref{eq:direct_problem_ovc}
are given and satisfy some restrictions which will be specified later. 
The differential notation is the standard ones, i.e.  the
symbols $\bu_t,\bw_t$ and $\rho_t$ denote the time derivatives and
$\nabla,\Delta,\diver$ and $\curl$ denote the gradient, Laplacian,
divergence and rotational operators, respectively.
Now, by applying the Helmholtz decomposition to $\bF$ 
and assuming a coherent form of $\bG$ we deduce that  
the vector fields $\bF$ and $\bG$ are representable by the following relations
\begin{eqnarray}
\bF(\bx,t)=f(t)(\nabla h(\bx,t)- \bm(\bx,t)), 
\quad
\bG(\bx,t)=g(t)\bq(\bx,t),
\quad
\mbox{in $\Omega_T$,}
	\label{eq:helmholtz_f_uno}
\end{eqnarray}
where $\bm$ and $\bq$ are given functions and $f,g$ and $h$
are unknown functions such that
\begin{eqnarray}
&&\diver(\rho\nabla h)=\diver(\rho \bm), 
	\quad\mbox{in}\quad\Omega, 
	\label{eq:helmholtz_f_dos}
\\
&&\frac{\partial h}{\partial \bn}=\bm\cdot  \bn, 
	\quad\mbox{on}\quad \Sigma_T,
	\label{eq:helmholtz_f_tres}
\\
&&\int_{\Omega}h(\bx,t)d\bx=0,
	\quad t\in[0,T],
	\label{eq:helmholtz_f_cuatro}
\end{eqnarray}
where $\bn$ is the outward unit normal vector to $\partial\Omega$.
We note that,  the determination
of $f$ and $g$, called the coefficients of $\bF$ and $\bG$, solves the inverse problem, 
since  $\bm$ and $\bq$ are known. Indeed, if $f$  is  determined,
we can find $h$ by 
solving \eqref{eq:helmholtz_f_dos}-\eqref{eq:helmholtz_f_cuatro}, since
$\bm$ is given.
Now, if additionally $g$ is determined, it is clear that $\bF$ and $\bG$
can be recovered by \eqref{eq:helmholtz_f_uno}.
Hence, from \eqref{eq:helmholtz_f_uno}, 
the inverse problem of recovery the vector fields $\bF$ and $\bG$  admits an
equivalent interpretation like an inverse coefficients determination problem.
Moreover, this behavior implies that the inverse problem can be equivalently
reformulated as an operator equation of second kind 
(see subsection~\ref{section:inv_operator_form}).

Similar inverse problems have been extensively  studied by
Prilepko, Orlovsky, Vasin and collaborators, 
see the monograph~\cite{pripleko_2000} and references therein. 
In a broad sense, their methodology to 
analyze the different source inverse problems consists of three big steps:
\begin{itemize}
\item[E$_1$.] To introduce the functional framework where the direct problem
is well-posed.
\item[E$_2$.] To use the definition of the solution of the direct
problem, in order to derive an operator equation of the second kind 
which solvability is equivalent to the solution of the inverse problem.
\item[E$_3$.] To prove the unique 
solution of the operator equation by applying fixed point arguments.
\end{itemize}
For instance, in order to fix ideas, let us
consider the following inverse source problem
for the heat equation: given $g,\omega$ and $\phi$ find a pair
of functions $\{u,f\}$ such that 
\begin{eqnarray}
\left\{
\begin{array}{lcll}
u_t(x,t)-\Delta u(x,t)&=&f(t)g(x,t),& (x,t)\in Q_T,
\\
u(x,0)&=&0,& x\in \Omega,
\\
u(x,t)&=&0,& (x,t)\in\partial\Sigma_T,
\\
\int_{\Omega}u(x,t)w(x)dx&=&\varphi(x), & t\in [0,T].
\end{array}
\right.
\label{eq:parabolic}
\end{eqnarray}
In this case for the step E$_1$,
is well known that one of the most general
functional frameworks where the well-posedness of the forward problem 
holds is given
by the Lebesgue and Sobolev spaces~\cite{evans_book}, 
see subsection~\ref{subsec:fucional_framework}  or 
consult~\cite{lions_book_1,lions_book_2,pripleko_2000} for the standard notation
of this spaces. Actually, considering that
$g\in C([0,T],L^2(\Omega))$ and $f\in L^2(\Omega)$,
we can prove that there exists $u\in W^{2,1}_{2,0}(Q_T)$
a unique generalized solution of the forward problem
\eqref{eq:parabolic} (i.e. except the overdetermination condition).
In the second step E$_2$, 
we prove that the solvability of the inverse problem
\eqref{eq:parabolic} is equivalent to the solution of the second
kind operator equation $f=Af+\psi$ on $L^2(0,T),$ where the operator $A$ and
the function  $\psi$ are defined as follows
\begin{eqnarray}
 (Af)(t)=\frac{1}{g_1(t)}\int_{\Omega}u(x,t)\Delta w(x)dx,
 \;\mbox{and}\;
 \psi =\phi' g_1
\;\mbox{with}\;
 g_1(t)=\int_{\Omega}g(x,t)w(x)dx.
\end{eqnarray}
Now, concerning to the step E$_3$ 
using the fixed point arguments prove that $f=Af+\psi$ has
a unique solution. In point of fact, we can prove that 
the operator equation can be rewritten as follows $f=\hat{A}f$
with $\hat{A}f=Af+\phi'/g_1$ and $\hat{A}$ has a fixed point
in $L^2(0,T).$ In general, the fact E$_3$  
is a straightforward consequence of a priori estimates for
the solution of the direct problem.
Clearly, the central difficulty in 
the analysis of source problems,
by applying this approach, is the proof of that the operator satisfies
the hypothesis of the fixed point theorem (see Theorem~\ref{teo:tikhonov}).
Now, we note that, it was proved  in \cite{pripleko_2000}
that  this general methodology 
can be  applied to analyze the inverse source problems
for  elliptic, hyperbolic, parabolic and
even for nonstationary linearized Navier-Stokes system with constant density
(in this case also consult \cite{vasin_1993}).
Recently, Fan and Nakamura  in \cite{fan_2008}, following
the same general approach and the results of J. Simon~\cite{simon_1990},
have been proved the local solvability of the source problem 
for Navier-Stokes system with variable density function.
In the present paper, we generalize these results for the 
inverse problem \eqref{eq:momento_lineal}-\eqref{eq:helmholtz_f_cuatro}.

Other general approaches applied to close inverse problems are given in the following
fundamental books on inverse problems~\cite{anger_book,isakov_book,isakov_book2}.
In our knowledge many of these techniques have not been
applied to analyze inverse source problems for
Navier-Stokes and related systems. An exception, on this sense
is the recent results obtained by Choulli and collaborators in~\cite{chouli_2013}
for the case of  the non-stationary linearized Navier–Stokes equations where
as observation data it is assumed that the velocity is 
given on an arbitrarily fixed sub-domain and over some time
interval. We remark that
on \cite{chouli_2013}, in order to solve the inverse problem,
the  authors have used an approach based on Carleman estimates
which was largely applied to geometric inverse problems~\cite{fernandez_2013}.

The paper is organized as follows. In section~\ref{sec:preliminares},
we recall some preliminary concepts and cover the
step E$_1$ of the methodology. In section~\ref{sec:priori_estimates},
we develop some a priori estimates which will be useful
in the step E$_3$. In section~\ref{sec:dp},
we prove the well posedness of the direct problem, i.e. we develop
the step  E$_2$. 
Finally, in
section~\ref{sec:ip} we set the details of the step E$_3$,
obtaining the local solvability of the inverse 
problem \eqref{eq:momento_lineal}-\eqref{eq:helmholtz_f_cuatro}.

\section{Preliminaries}
\label{sec:preliminares}

In this section we consider the notation of the functional framework,
the rigorous definition of direct and inverse problems and
the general assumptions.

\subsection{Some functional spaces}
\label{subsec:fucional_framework}
Let us starting by recalling the standard notation of some functional
spaces and operators which are familiar in the mathematical 
theory of fluids modelled
by Navier-Stokes system, see \cite{antontsev_1990,lions_book_1,lions_book_2,temam_1977}. 
The Banach space of measurable functions that are $p$-integrable 
in the sense of Lebesgue or are essentially bounded on $\Omega$ are denoted by $L^p(\Omega)$
for $p\in[1,\infty[$ and by $L^{\infty}(\Omega)$, respectively. 
We recall that, the norms in $L^p(\Omega)$ for  $p\in [1,\infty[$ and
$p=\infty$ are defined as follows
\begin{eqnarray*}
\|u\|_{L^p(\Omega)}:=\left(\int_{\Omega}|u(\bx)|^pd\bx\right)^{1/p}
\mbox{ and }
\|u\|_{L^{\infty}(\Omega)}:=\esssup_{\bx\in\Omega}|u(\bx)|,
\end{eqnarray*}
respectively.
The notation $W^{m,q}(\Omega),$ where $m\in\mathbb{N}$ and $q\ge 1$ is used for the Sobolev
space consisting of all functions in $L^q(\Omega)$ having all distributional derivatives
of the first $m$ orders belongs to $L^q(\Omega)$, i.e.
\begin{eqnarray*}
W^{m,q}(\Omega):=\Big\{\;u\in  L^q(\Omega)\;:\; 
	D^{\alpha}u\in L^q(\Omega)\mbox{ for } |\alpha|=1,\ldots,m\;\Big\}.
\end{eqnarray*}
The norm of $W^{m,q}(\Omega)$ is naturally defined as follows
\begin{eqnarray*}
\|u\|_{W^{m,q}(\Omega)}
	:=\left(\sum_{k=0}^{m}\sum_{|\alpha|=k}
	\|D^{\alpha}u\|^q_{L^q(\Omega)}\right)^{1/q}
\mbox{ and }
\|u\|_{W^{m,\infty}(\Omega)}
	:=\max_{0\le |\alpha|\le m}\|D^{\alpha}u\|_{L^\infty(\Omega)}.
\end{eqnarray*}
The vector-valued spaces $[L^2(\Omega)]^3$ and
$[W^{m,p}(\Omega)]^3$ are defined as usual and by simplicity are
denoted by bold characters, i.e. $\bL^2(\Omega)$ and
$\bW^{m,p}(\Omega)$, respectively.
Also, we use the following rather common notation 
in mathematical theory of fluid mechanics:
\begin{eqnarray*}
&&H^m(\Omega)=W^{m,2}(\Omega), 
\quad
H^m_0(\Omega)=\overline{C_0^\infty(\Omega)}^{\|\cdot\|_{H^1(\Omega)}},
\\
&&
\mathcal{V}(\Omega)=\Big\{\bv\in (C_0^\infty(\Omega))^3\;:\;
\diver(\bv)= 0\mbox{ in }\Omega\;\Big\},
\\
&&
\bH=\overline{\mathcal{V}(\Omega)}^{\|\cdot\|_{\bL^2(\Omega)}}
\quad
\mbox{and}
\quad
\bV=\overline{\mathcal{V}(\Omega)}^{\|\cdot\|_{\bH_0^1(\Omega)}},
\end{eqnarray*}
where $\overline{A}^{\|\cdot\|_{B}}$ denotes the closure of $A$ in $B$.
Furthermore, for a given Banach space $X$, we denote by $L^r(0,T;X)$, $r\ge 1$,
the  Banach space of the $X$-valued functions having bounded the
norm $\|\cdot\|_{L^r(0,T;B)}$ defined as follows
\begin{eqnarray*}
\|u\|_{L^r(0,T;B)}:=\left(\int_0^T\|u(\cdot,t)\|_{B}^r dt\right)^{1/r}
\mbox{ and }
\|u\|_{L^{\infty}(0,T;B)}:=\esssup_{t\in[0,T]}\|u(\cdot,t)\|_{B}.
\end{eqnarray*}
Concerning to the linear operators, we define the operators: $A,L_0$
and $L$.
We denote by $A$ the stokes operator defined from 
$D(A):=\bV\cap \bH^2(\Omega)\subset \bH$ to $ \bH$
by $A\bv=P(-\Delta \bv)$, where $P$ is the orthogonal projection of 
$\bL^2(\Omega)$ onto
$\bH$ induced by the Helmholtz decomposition of $\bL^2(\Omega)$. It is well known that
$A$ is an unbounded linear and positive self-adjoint  operator, and is characterized by 
the following identity
\begin{eqnarray}
(A\bw,\bv)=(\nabla \bw,\nabla \bv),
	\quad \forall \bw\in D(A),\quad \bv\in V,
\label{eq:ident_stokes}
\end{eqnarray}
where $(\cdot,\cdot)$ is the usual scalar product in $\bL^2(\Omega)$.
In second place, we consider the strongly uniformly elliptic operators
$L_0$ and $L$ defined on $D(L_0)=D(L)=\bH^1_0(\Omega)\cap \bH^2(\Omega)$
as follows
\begin{eqnarray}
L_0z=-(c_a+c_d)\Delta z-(c_0+c_d-c_a)\nabla \diver z
\quad\mbox{and}\quad
Lz=L_0z+4\mu_r z.
\label{eq:def_of_L_and_L0}
\end{eqnarray}
Note that $L$ is a positive operator under the assumption $c_0+c_d>c_a$, see 
\eqref{eq:momento_angular}.

\subsection{Direct and inverse problem solution definitions}
Using the described notation on the section~\ref{subsec:fucional_framework}
and the ideas given on \cite{boldrini_2003},
we can give a rigorous formulation of direct and inverse problem related to
equations \eqref{eq:momento_lineal}-\eqref{eq:helmholtz_f_cuatro}.

\begin{definition}
\label{def:strong_solutions_dp}
Consider that the  functions $f,g,m$ and $q$ are given.
Then,
a collection of functions
$\{\bu,\bw,\rho,p,h\}$ is a 
solution of the direct problem \eqref{eq:momento_lineal}-\eqref{eq:direct_problem_ic}
and \eqref{eq:helmholtz_f_uno}-\eqref{eq:helmholtz_f_cuatro} if
there exists $T_*\in ]0,T]$ such that the functions satisfy
the following four conditions
\begin{enumerate}
\item[(a)] Regularity conditions:
\begin{eqnarray}
&&
\bu\in C^0\Big([0,T_*[;D(A)\Big)\cap C^1\Big([0,T_*[;\bH\Big),
\label{eq:reg_of_u}
\\
&&
\bw\in C^0\Big([0,T_*[;D(L)\Big)\cap C^1\Big([0,T_*[;\bL^2(\Omega)\Big)
\quad\mbox{and}
\label{eq:reg_of_w}
\\
&&
\rho \in C^1(\overline{\Omega}\times [0,T_*[).
\label{eq:reg_of_rho}
\end{eqnarray}
\item[(b)] Integral identities:
\begin{eqnarray}
&&\Big((\rho \bu)_t,\bv\Big)
	-\int_{\Omega}\rho \bu\otimes \bu:\nabla \bv\; d\bx
	+(\mu +\mu_r)\Big(A \bu,\bv \Big)
\nonumber\\
&&
\hspace{2cm}
=2\mu_r \Big(\curl \bw,\bv\Big)+
	\Big(\rho \bF,\bv\Big),\quad
	\mbox{for } t\in ]0,T_*[\mbox{ and }\forall \bv\in V,
\label{eq:formulacion_variacional_u}
\\
&&\Big((\rho \bw)_t,\bvarphi\Big)
	-\int_{\Omega}\rho \bw\otimes \bw:\nabla \bvarphi \; d\bx
	+\Big(L \bw,\bvarphi \Big)
\nonumber\\
&&
\hspace{2cm}
=2\mu_r \Big(\curl \bu,\bvarphi\Big)+
	\Big(\rho \bG,\bvarphi\Big),\quad
	\mbox{for } t\in ]0,T_*[\mbox{ and }\forall \bvarphi\in 
	\bH^1_0(\Omega).
\label{eq:formulacion_variacional_w}
  \end{eqnarray}
\item[(c)] Mass conservation: $\rho$ satisfies \eqref{eq:ecuacion_continuidad} for
$(x,t)\in \overline{\Omega}\times [0,T_*[$.
\item[(d)] Initial condition:
$\bu,\bw,\rho$ satisfies \eqref{eq:direct_problem_ic} for $\bx\in\Omega.$
\end{enumerate}

\end{definition}

\begin{definition}
\label{def:strong_solutions_ip}
Consider that the  functions $\phi^{\bu},\phi^{\bw},\bpsi^{\bu},\bpsi^{\bw}$ $m$ and $q$ are given.
Then, a collection of functions $\{\bu,\bw,\rho,p,h,f,g\}$
is called a solution of the inverse problem
\eqref{eq:momento_lineal}-\eqref{eq:helmholtz_f_cuatro} if the following 
three conditions hold:
\begin{enumerate}
\item[(i)] The functions $f$ and $g$ are belongs to $H^1([0,T])$, 
\item[(ii)] The collection $\{\bu,\bw,\rho,p,h\}$ is a solution of the direct 
problem \eqref{eq:momento_lineal}-\eqref{eq:direct_problem_ic}
and \eqref{eq:helmholtz_f_uno}-\eqref{eq:helmholtz_f_cuatro} and
\item[(iii)] The overdetermination condition \eqref{eq:direct_problem_ovc}
is satisfied.
\end{enumerate}

\end{definition}

\subsection{General hypothesis and well-posedness of the direct problem}

Hereafter, we make the following regularity assumptions:
\begin{enumerate}
\item[(H$_1$)] The initial  density $\rho_0$  belongs to $C^1(\overline{\Omega})$
and $\rho_0(\bx)\in [\alpha,\beta]\subset]0,\infty[$ a.e. in $\Omega$,
\item[(H$_2$)] The initial velocity $\bu_0$  belongs to $D(A)$ and
\item[(H$_3$)] The initial angular velocity $\bw_0$   belongs to $D(L).$
\item[(H$_4$)] The functions $h$ and $r$ belong to $C^1([0,T])$.
\item[(H$_5$)] The functions $\bpsi^{\bu}$ and $\bpsi^{\bw}$ belong to 
	$\bH^1_0(\Omega)\cap \bW^{1,\infty}(\Omega)\cap \bH^2(\Omega)$ and such that
	$\diver (\bpsi^{\bu})=0$ in $\Omega.$
\item[(H$_6$)] The functions  $\phi^{\bu}$ and $\phi^{\bw}$ belong to 
	$\Big(H^2(\Omega)\Big)^3.$
\item[(H$_7$)] There exists $h^\epsilon,r^\epsilon\in\R^+$ such that
\begin{eqnarray*}
\left|\int_\Omega \rho_0(\bx)\Big(\nabla h(\bx,0)-\bm(\bx,0)\Big)d\bx\right|\ge 2h^\epsilon
\quad\mbox{and}\quad
\left|\int_\Omega \rho_0(\bx)\bq(\bx,0)d\bx\right|\ge 2r^\epsilon.
\end{eqnarray*}
\end{enumerate}

\section{A priori estimates}
\label{sec:priori_estimates}

The main result of this section is the following theorem: 

\begin{theorem}
\label{teo:global_estimates}
Let $\upeta\in\R^+$ such that $\|(f,g)\|_{[H^1(0,T)]^2}\le \upeta$.
Then, there exists $\upkappa_j\in\R^+$ for $j=1,\ldots,11$ and 
two small enough times $T_1,T_2\in [0,T_*]$, 
independents of $f$ and $g$, such that the following estimates hold
\begin{eqnarray}
&&\| \bu\|_{L^\infty([0,T_1];\bH^1_0(\Omega))}+
	\| \bw\|_{L^\infty([0,T_1];\bH^1_0(\Omega))}+
	\|\bu_t\|_{L^\infty([0,T_1];\bL^2(\Omega))}+
	\| \bw_t\|_{L^\infty([0,T_1];\bL^2(\Omega))}
	\le\upkappa_1,\qquad
\label{teo:global_estimates_kapa1}
\\
&&
\| \bu_t\|_{L^\infty([0,T_2];\bL^2(\Omega))}+
	\| \bw_t\|_{L^\infty([0,T_2];\bL^2(\Omega))}+
	\|\bu_t\|_{L^2([0,T_2];\bH^1(\Omega))}
	\nonumber
\\
&&
\hspace{3.1cm}
	+
	\| \bw_t\|_{L^2([0,T_2];\bH^1(\Omega))}+
	\| \nabla \rho\|_{L^\infty([0,T_*];\bL^q(\Omega))}
	\le\upkappa_2,
\label{teo:global_estimates_kapa2}
\\
&&
	\| h\|_{L^\infty([0,T_2];H^2(\Omega))}\le\upkappa_3,
\label{teo:global_estimates_kapa3_4}
\\
&&
	\| \nabla h_t\|_{L^\infty([0,T_2];\bL^2(\Omega))}\le\upkappa_4+\upkappa_5 \upeta,
\label{teo:global_estimates_kapa5_8}
\\
&&
\| \bu\|_{L^\infty([0,T_1];\bH^2(\Omega))}+
	\| p\|_{L^\infty([0,T_1];H^1(\Omega))}
	\le\upkappa_6+\upkappa_{7} \upeta,
\label{teo:global_estimates_kapa9_10}
\\
&&
\| \rho_t\|_{L^\infty([0,T_*];L^2(\Omega))}\le \upkappa_{8},
	\qquad
	\| \rho_t\|_{L^\infty([0,T_*];L^q(\Omega))}
	\le\upkappa_{9}+\upkappa_{10} \upeta\quad\mbox{and}
\label{teo:global_estimates_kapa11_13}
\\
&&
\| \bu\|_{L^2([0,T_2];\bW^{2,s}(\Omega))}+
	\| p\|_{L^2([0,T_2];W^{1,s}(\Omega))}
	\le \upkappa_{11}.
\label{teo:global_estimates_kapa14}
\end{eqnarray}
\end{theorem}

\vspace{0.5cm}
We postpone the proof of Theorem~\ref{teo:global_estimates} to the 
subsection~\ref{subsec:proof_of_teo:global_estimates},
since we require the introduction of   
several  appropriate space estimates which leads to the inequalities 
\eqref{teo:global_estimates_kapa1}-\eqref{teo:global_estimates_kapa14}.

\subsection{Some space a priori estimates}
\label{subsec:a_priori_estimates}

We recall that, by
Pioncar\'e and the Gagliardo-Nirenberg inequalities we have that
there exists  $C_{poi}>0$ and $C_{gn}>0$ depending only on 
$p,q$ and $\Omega$ such that
\begin{eqnarray}
\hspace{0.5cm}
\begin{array}{l}
\|u\|_{L^q(\Omega)}\le C_{poi} \|\nabla u\|_{L^p(\Omega)},\;\;
p\in [1,3[,\;\; q\in [1,3p(3-p)^{-1}]\;\;\mbox{and}\;\;
u\in W^{1,p}_0(\Omega),
\\
\|\nabla u\|_{L^{2q/(q-2)(\Omega)}}
\le C_{gn} \|\nabla u\|^{1-3/q}_{L^2(\Omega)}\|u\|^{3/q}_{H^2(\Omega)},
\;\;
u\in H^2(\Omega).
\end{array}
\label{eq:spoingag}
\end{eqnarray}
Also we denote by $C^{2,\infty}_{iny}$ the constant such that
$\| u\|_{L^\infty(\Omega)}\le C^{2,\infty}_{iny} \| u\|_{H^2(\Omega)}$
for all $u\in H^2(\Omega)$,
i.e. for the continuous embedding of $H^2(\Omega)$ in $L^\infty(\Omega)$.

\begin{lemma}
\label{lema:aprioriestimatesforrho}
The following estimate holds:
$0< \alpha \le \rho(\bx,t)\le\beta$ for 
all $(\bx,t)\in\Omega\times [0,T_*]$.
\end{lemma}

\begin{proof}
We deduce the estimate by equations \eqref{eq:incompresibilidad}
and \eqref{eq:ecuacion_continuidad}, the hypothesis (H$_1$)
and the maximum principle.
\end{proof}

\begin{lemma}
\label{lema:aprioriestimatesforhandr}
Assume that $\{\epsilon,\epsilon',\epsilon'',\dot{\epsilon}\}\subset \mathbb{R}^+$ 
with $\epsilon\in ]0,2\alpha\beta^{-1}[$ and
$\epsilon'\in ]0,\alpha(C^{reg}C_{gn})^{-1}[$. Then,
there exists  $\Uppi_i,$  $i=1,\ldots,4,$ defined as follows
\begin{eqnarray}
&&\Uppi_1=\sqrt{\beta\epsilon^{-1} (2\alpha-\beta \epsilon)^{-1}},
\label{eq:uppi_uno}
\\
&&\Uppi_2=\frac{C^{reg}}{\alpha-\epsilon'C^{reg}C_{gn}}\Big(
    \alpha\|\nabla \bm(\cdot,t)\|_{\bL^2}
    +|\Omega|^{(q-2)/q}\epsilon'\|\bm(\cdot,t)\|_{\bL^\infty}
    \Big),
\label{eq:uppi_dos}    
 \\
&&\Uppi_3=\frac{C^{reg}}{\alpha-\epsilon'C^{reg}C_{gn}}\Big(
        |\Omega|^{(q-2)/q}\epsilon'(\epsilon'')^{3/(3-q)}
        +C_{gn}(\epsilon')^{3/(3-q)}\Uppi_1\|\nabla \bm(\cdot,t)\|_{\bL^2}
        \Big),
\label{eq:uppi_tres}
\\
&&\Uppi_4=
\frac{C_{gn}C^{2,\infty}_{iny}\dot{\epsilon}}{\alpha}
 \| \bu(\cdot,t)\|_{\bH^2}\|\nabla\rho(\cdot,t)\|_{L^q}
 \quad\mbox{and}
 \label{eq:uppi_cuatro}
 \\
&&\Uppi_5=
\frac{C^{2,\infty}_{iny}}{\alpha}\Big(C_{gn}
 (\dot{\epsilon})^{3/(3-q)}\Uppi_1\| \bu(\cdot,t)\|_{\bH^2}
 \| \bm(\cdot,t)\|_{\bL^2} +
 C_{poi}\| \bm(\cdot,t)\|_{\bH^2}\|\nabla \bu(\cdot,t)\|_{\bL^2}
 \Big),\qquad
\label{eq:uppi_cinco}
\end{eqnarray}
such that the following estimates holds
\begin{eqnarray}
&&\|\nabla h(\cdot,t)\|_{\bL^2}\le
\Uppi_1\;\|\bm(\cdot,t)\|_{\bL^2},
\label{hx_and_rx_estimate}
\\
&&\|h (\cdot,t)\|_{H^2}
 \le 
 \Uppi_2+\Uppi_3 \|\nabla\rho(\cdot,t)\|^{q/(q-3)}_{\bL^q},
\label{h_and_r_estimate}
\\
&&\|\nabla h_t(\cdot,t)\|_{\bL^2}\le 
\frac{\beta}{\alpha}\|\bm_t(\cdot,t)\|_{\bL^2}
+\Uppi_4
 \Big(\Uppi_2+\Uppi_3 \|\nabla\rho(\cdot,t)\|^{q/(q-3)}_{\bL^q}\Big)
 +\Uppi_5 \|\nabla\rho(\cdot,t)\|_{\bL^q},
 \qquad
\label{ht_and_rt_estimate}
\end{eqnarray}
for all $t\in [0,T_*]$.
Hereafter, the notation $\|\cdot\|_{L^p}$ and $\|\cdot\|_{H^p}$   
are  used to abbreviate
the norm $\|\cdot\|_{{L^p}(\Omega)}$ and $\|\cdot\|_{{H^p}(\Omega)}$,
respectively.
\end{lemma}

\begin{proof}
Form \eqref{eq:helmholtz_f_dos}, by applying 
Lemma~\ref{lema:aprioriestimatesforrho}, integration by parts, the boundary condition
\eqref{eq:helmholtz_f_tres}, and H\"older and Young inequalities, we
have that
\begin{eqnarray*}
\alpha\int_{\Omega}|\nabla h|^2(\bx,t)d\bx
&\le & \int_{\Omega}\Big(\rho|\nabla h|^2\Big)(\bx,t)d\bx
=\int_{\Omega}\Big(\rho\nabla h\cdot \bm\Big)(\bx,t)d\bx
\\
&\le & \beta\|\nabla h(\cdot,t)\|_{\bL^2}\|\bm(\cdot,t)\|_{\bL^2}
\le \frac{\epsilon\beta}{2}\|\nabla h(\cdot,t)\|^2_{\bL^2}
+\frac{\beta}{2\epsilon}\|\bm(\cdot,t)\|^2_{\bL^2},
\end{eqnarray*}
for each $t\in [0,T_*]$. Hence, we see that \eqref{hx_and_rx_estimate}
holds for $\Uppi_1$ defined by \eqref{eq:uppi_uno}.

Now, we can proceed to prove \eqref{h_and_r_estimate}. We start  recalling
the identities $\diver(\rho \nabla h)=\rho \Delta h+\nabla\rho\cdot\nabla h$ and
$\diver(\rho \bm)=\rho\diver(\bm)+\nabla\rho\cdot \bm$, which imply that 
the  equation  \eqref{eq:helmholtz_f_dos} can be rewritten as follows
\begin{eqnarray}
 \Delta h=\diver(\bm)+\frac{1}{\rho}\nabla\rho\cdot \bm
 -\frac{1}{\rho}\nabla\rho\cdot\nabla h.
 \label{eq:elliptic_helmholtz}
\end{eqnarray}
Clearly, by  the estimate \eqref{hx_and_rx_estimate} we deduce
that the right hand side of \eqref{eq:elliptic_helmholtz}
belongs to $L^2(\Omega)$. Then, by
the regularity of solutions for \eqref{eq:elliptic_helmholtz},
the inequality \eqref{eq:spoingag},
Lemma~\ref{lema:aprioriestimatesforrho}
and  the estimate \eqref{hx_and_rx_estimate}, 
we can follow that there exists $C_{reg}>0$ independent of $h$ such that 
the following bound
\begin{eqnarray*}
&& \|h (\cdot,t)\|_{H^2}
 \le C_{reg}
  \left\|\left(\diver(\bm)+\frac{1}{\rho}\nabla\rho\cdot \bm
  -\frac{1}{\rho}\nabla\rho\cdot\nabla h\right)(\cdot,t)\right\|_{\bL^2}
  \\
  &&\hspace{1cm}
  \le C_{reg}\Big\{
  \|\nabla \bm(\cdot,t)\|_{\bL^2}
  +\frac{1}{\alpha}\| \bm(\cdot,t)\|_{L^{\infty}}\|\nabla\rho(\cdot,t)\|_{\bL^2}
  +\frac{1}{\alpha}\|\nabla\rho(\cdot,t)\|_{\bL^q}\|\nabla h(\cdot,t)\|_{\bL^{2q/(q-2)}}
  \Big\}
  \\
  &&\hspace{1cm}
  \le
  C_{reg}\Big\{
  \|\nabla \bm(\cdot,t)\|_{\bL^2}
  +\frac{|\Omega|^{\frac{q-2}{2q}}}{\alpha}\| \bm(\cdot,t)\|_{\bL^{\infty}}
  \|\nabla\rho(\cdot,t)\|_{\bL^q}
  \\
  &&\hspace{2.5cm}
  +\frac{C_{gn}}{\alpha}\|\nabla\rho(\cdot,t)\|_{\bL^q}\|
     \|\nabla h(\cdot,t)\|^{1-3/q}_{\bL^2}\|h(\cdot,t)\|^{3/q}_{H^2}
     \Big\}
  \\
  &&\hspace{1cm}
  \le
  C_{reg}\Big\{
  \|\nabla \bm(\cdot,t)\|_{\bL^2}
  +\frac{|\Omega|^{\frac{q-2}{2q}}}{\alpha}\| \bm(\cdot,t)\|_{\bL^{\infty}}
  \|\nabla\rho(\cdot,t)\|_{\bL^q}
  \\
  &&\hspace{2.5cm}
  +\frac{C_{gn}(\Uppi_1)^{1-3/q}}{\alpha}\;
  \|\bm(\cdot,t)\|^{1-3/q}_{\bL^2}
  \|\nabla\rho(\cdot,t)\|_{\bL^q}\|h(\cdot,t)\|^{3/q}_{H^2}\Big\},
\end{eqnarray*}
holds for each $t\in [0,T_*]$. Thus, by the application
of two times of the Young's inequality we complete the proof
of \eqref{h_and_r_estimate} with $\Uppi_2$ of the form
given in \eqref{eq:uppi_dos}.

The proof of \eqref{ht_and_rt_estimate} is given as follows.
Taking $\partial_t$ to 
the first equation of \eqref{eq:helmholtz_f_dos}, testing the result by $h_t$,
using the estimate of Lemma~\ref{lema:aprioriestimatesforrho},
the H\"older inequality, the equation \eqref{eq:ecuacion_continuidad}
and inequality \eqref{eq:spoingag},
we have that
\begin{eqnarray*}
&&\alpha\|\nabla h_t(\cdot,t)\|^2_{\bL^2}
\le 
 \int_{\Omega} \rho|\nabla h_t (\cdot,t)|^2d\bx
 \\
&& \qquad
 \le 
 \left|\int_{\Omega} \Big(\rho \bm_t \cdot \nabla h_t\Big)(\cdot,t)d\bx\right|
 +\left|\int_{\Omega} \Big(\rho_t \nabla h\cdot \nabla h_t\Big)(\cdot,t)d\bx\right|
 +\left|\int_{\Omega} \Big(\rho_t   \bm\cdot \nabla h_t\Big)(\cdot,t)d\bx\right|
 \\
&& \qquad
 \le 
 \|\nabla\rho(\cdot,t)\|_{\bL^\infty}\|\bm_t(\cdot,t)\|_{\bL^2}
 \|\nabla h_t(\cdot,t)\|_{\bL^2}
 +\left|\int_{\Omega} \Big((\bu\cdot\nabla\rho) 
 \nabla h\cdot \nabla h_t\Big)(\cdot,t)d\bx\right|
 \\
 &&\qquad\quad
 +\left|\int_{\Omega} \Big((\bu\cdot\nabla\rho)
      \bm\cdot \nabla h_t\Big)(\cdot,t)d\bx\right|
 \\ 
 &&\qquad
  \le 
 \beta\|\bm_t(\cdot,t)\|_{\bL^2}\|\nabla h_t(\cdot,t)\|_{\bL^2}
 +\| \bu(\cdot,t)\|_{\bL^\infty}\|\nabla\rho(\cdot,t)\|_{\bL^q}
 \|\nabla h(\cdot,t)\|_{\bL^{2q/(q-2)}}\|\nabla h_t(\cdot,t)\|_{\bL^2}
  \\ 
 &&\qquad\quad
 +\| \bm(\cdot,t)\|_{\bL^\infty}\|\nabla\rho(\cdot,t)\|_{\bL^q}
 \|\bu(\cdot,t)\|_{\bL^{2q/(q-2)}}\|\nabla h_t(\cdot,t)\|_{\bL^2}
 \\
&& \qquad
  \le 
 \|\nabla h_t(\cdot,t)\|_{\bL^2}\Big\{\beta\|\bm_t(\cdot,t)\|_{\bL^2}
 +C_{gn}\| \bu(\cdot,t)\|_{\bL^\infty}\|\nabla\rho(\cdot,t)\|_{\bL^q}
 \|\nabla h(\cdot,t)\|^{1-3/q}_{\bL^{2}}\|h(\cdot,t)\|^{3/q}_{H^{2}}
  \\ 
&& \qquad\quad
 +\| \bm(\cdot,t)\|_{\bL^\infty}\|\nabla\rho(\cdot,t)\|_{\bL^q}
 \|\bu(\cdot,t)\|_{\bL^{2q/(q-2)}}\Big\}.
\end{eqnarray*}
Then, by \eqref{eq:spoingag}, the continuous inclusion of $H^2$ in $ L^\infty$, 
 and the Young inequality, we obtain
\begin{eqnarray*}
\|\nabla h_t(\cdot,t)\|_{\bL^2} 
  \le 
\frac{\beta}{\alpha}\|\bm_t(\cdot,t)\|_{\bL^2}
 +\frac{C_{gn} C^{2,\infty}_{iny}\dot{\epsilon}}{\alpha}
 \| \bu(\cdot,t)\|_{\bH^2}\|\nabla\rho(\cdot,t)\|_{\bL^q}
 \|h(\cdot,t)\|_{H^{2}}+\frac{C^{2,\infty}_{iny}}{\alpha}\times
 \\
 \hspace{1cm}
 \Big(C_{gn}(\dot{\epsilon})^{3/(3-q)}
 \| \bu(\cdot,t)\|_{\bH^2}\|\nabla h(\cdot,t)\|_{\bL^q}
 +C_{poi}\dot{\epsilon}
 \| \bm(\cdot,t)\|_{\bH^2}\|\nabla \bu(\cdot,t)\|_{\bL^2}\Big)
 \|\nabla\rho(\cdot,t)\|_{\bL^q},
\end{eqnarray*}
which implies \eqref{ht_and_rt_estimate} by 
straightforward application of \eqref{hx_and_rx_estimate} and
\eqref{h_and_r_estimate}.
\end{proof}

\begin{lemma}
\label{prop:auxiliar_stokes_elliptic_results}
There exists  $\Upsilon_i\in \R^+$ for $i\in\{1,\ldots,5\}$,
depending only on $\Omega,c_a,c_0,c_d,\alpha,\beta$ and $\mu_r$
(independent of $f$ and $g$), such that the following estimate holds:
\begin{eqnarray} 
&&\Upsilon_1\|\bu(\cdot,t)\|_{\bH^2}+\|p(\cdot,t)\|_{H^2}+\Upsilon_1\|\bw(\cdot,t)\|_{\bH^2}
\nonumber
\\
&&
\quad
\le
\Upsilon_2
\Big[\|\nabla \bu(\cdot,t)\|^{3}_{\bL^2}+\|\nabla \bw(\cdot,t)\|^{3}_{\bL^2}\Big]
+
\Upsilon_3
\Big[\|\nabla \bw(\cdot,t)\|_{\bL^2}+\|\nabla \bu(\cdot,t)\|_{\bL^2}\Big]
\nonumber
\\
&&
\qquad
+\Upsilon_4
\Big[|f(t)|\|\bm(\cdot,t)\|_{\bL^2(\Omega)}+|g(t)|\|\bq(\cdot,t)\|_{\bL^2(\Omega)}\Big]
\nonumber
\\
&&
\qquad
+
\Upsilon_5
\Big[
\|\big(\sqrt{\rho} \bu_t\big)(\cdot,t)\|_{\bL^2}
+\|\big(\sqrt{\rho} \bw_t\big)(\cdot,t)\|_{\bL^2}
\Big],
\label{eq:regularity_u_p_w}
\end{eqnarray}
for all $t\in[0,T_*]$.
\end{lemma}

\begin{proof}
The inequality  \eqref{eq:regularity_u_p_w} is a consequence of the regularity
of solutions for the Stokes system satisfied by $u$ and $p$ and the uniformly
elliptic  equation satisfied by $w$. Indeed, we first note that
the equations \eqref{eq:momento_lineal},
\eqref{eq:incompresibilidad} and \eqref{eq:helmholtz_f_uno}
imply that $\bu$ and $p$ satisfy the Stokes problem given by the equation
\begin{eqnarray}
-(\mu +\mu_r)\Delta \bu+\nabla p
	=2\mu_r \curl \bw+\rho f(\nabla h-\bm)
	-\rho \bu_t-\rho (\bu\cdot\nabla )\bu,
	\quad\mbox{in}\quad Q_T,
\label{eq:stokes_problem}
\end{eqnarray}
where the incompressibility condition is given by  \eqref{eq:incompresibilidad}
and the initial and boundary conditions are given by
\eqref{eq:direct_problem_ic} and \eqref{eq:direct_problem_bc}, respectively. 
Hence, by applying the result given in \cite{temam_1977}
for the regularity of the 
solutions for stokes equation, 
the Minkowski and H\"older  inequalities, we deduce that 
\begin{eqnarray}
&& \|\bu(\cdot,t)\|_{\bH^2}+\|p(\cdot,t)\|_{H^2}
\nonumber
\\
&&
\qquad
\le C^{reg}_1\; \Big[
2\mu_r \|\curl \bw(\cdot,t)\|_{\bL^2}
+\|\big(\rho f(\nabla h-\bm)\big)(\cdot,t)\|_{\bL^2}
+\|(\rho \bu_t)(\cdot,t)\|_{\bL^2}
\nonumber
\\
&&
\qquad\quad
+\|\rho ((\bu\cdot\nabla )\bu)(\cdot,t)\|_{\bL^2}\Big]
\nonumber
\\
&&
\qquad
\le C^{reg}_1\; \Big[
2\mu_r \|\nabla \bw(\cdot,t)\|_{\bL^2}
+|f(t)|\|\rho (\cdot,t)\|_{L^2}
\Big(\|\nabla h(\cdot,t)\|_{\bL^2}+\|\bm(\cdot,t)\|_{\bL^2}\Big)
\nonumber
\\
&&
\qquad\quad
+\|(\rho \bu_t)(\cdot,t)\|_{\bL^2}
+\|\rho(\cdot,t)\|_{L^2}\|\bu(\cdot,t)\|_{\bL^6}\|\nabla \bu(\cdot,t)\|_{\bL^3}\Big],
\label{eq:stokes_for_uuu}
\end{eqnarray}
where $C^{reg}_1$ is a positive constant depending on $\mu,\mu_r$
and $\Omega$. In the second place,
by \eqref{eq:momento_angular}, \eqref{eq:helmholtz_f_uno} and \eqref{eq:def_of_L_and_L0},
we deduce that $\bw$ satisfies the following equation
\begin{eqnarray}
L\bw
	=2\mu_r \curl \bu+\rho g \bq
	-\rho \bu_t-\rho (\bw\cdot\nabla )\bw,
	\quad\mbox{in}\quad Q_T.
\end{eqnarray}
Then by the regularity results for  the 
solutions for uniformly elliptic  equations (see for instance \cite{evans_book}), 
the Minkowski and H\"older  inequalities, and \eqref{eq:spoingag}, we have that 
\begin{eqnarray}
 \|\bw(\cdot,t)\|_{\bH^2}
&\le& C^{reg}_2\; \Big[
2\mu_r \|\curl \bu(\cdot,t)\|_{\bL^2}
+\|\big(\rho g\bq\big)(\cdot,t)\|_{\bL^2}
+\|(\rho \bw_t)(\cdot,t)\|_{\bL^2}
\nonumber
\\
&&
+\|\rho ((\bw\cdot\nabla )\bw)(\cdot,t)\|_{\bL^2}
+\|\bw(\cdot,t)\|_{\bL^2}\Big]
\nonumber
\\
&\le& C^{reg}_2\; \Big[
2\mu_r \|\nabla \bu(\cdot,t)\|_{\bL^2}
+|g(t)|\|\rho (\cdot,t)\|_{L^2}\|\bq(\cdot,t)\|_{\bL^2}
+\|(\rho \bw_t)(\cdot,t)\|_{\bL^2}
\nonumber
\\
&&
+\|\rho(\cdot,t)\|_{L^2}\|\bw(\cdot,t)\|_{\bL^6}\|\nabla\b w(\cdot,t)\|_{\bL^3}
+C_{poi}\|\nabla \bw(\cdot,t)\|_{\bL^2}\Big],
\label{eq:regularity_for_Lw}
\end{eqnarray}
where $C^{reg}_2$ is a positive constant depending only on $\Omega$ and 
on the coefficients of $L$.
Now, we note that the second terms on the right hand sides 
of \eqref{eq:stokes_for_uuu}
and \eqref{eq:regularity_for_Lw} can be bound
by application of 
Lemmas~\ref{lema:aprioriestimatesforrho}-\ref{lema:aprioriestimatesforhandr}
and \eqref{eq:spoingag}.
Hence, if we sum the bounded results,  we
obtain the following inequality
\begin{eqnarray}
&& \|\bu(\cdot,t)\|_{\bH^2}+\|p(\cdot,t)\|_{H^2}+\|\bw(\cdot,t)\|_{\bH^2}
\nonumber
\\&&
\qquad
\le
C_M\;\Big\{
(2\mu_r +C_{poi}) \;\|\nabla \bw(\cdot,t)\|_{\bL^2}+2\mu_r\|\nabla \bu(\cdot,t)\|_{\bL^2}
\nonumber
\\&&
\qquad\quad
+\beta |\Omega|^{1/2}
\Big(\Uppi_1+1\Big)
\;
\Big[|f(t)|\|\bm(\cdot,t)\|_{\bL^2}+|g(t)|\|\bq(\cdot,t)\|_{\bL^2}\Big]
\nonumber
\\&&
\qquad\quad
+
(\beta)^{1/2}
\Big[
\|\big(\sqrt{\rho} \bu_t\big)(\cdot,t)\|_{\bL^2}
+\|\big(\sqrt{\rho} \bw_t\big)(\cdot,t)\|_{\bL^2}
\Big]\nonumber
\\&&
\qquad\quad
+
\beta |\Omega|^{1/2}C_{gn}C_{poi}\;
\Big[\|\nabla \bu(\cdot,t)\|^{3/2}_{\bL^2}\|\bu(\cdot,t)\|^{1/2}_{\bH^2}
+\|\nabla \bw(\cdot,t)\|^{3/2}_{\bL^2}\|\bw(\cdot,t)\|^{1/2}_{\bH^2}\Big]\Big\},
\label{eq:regularity_u_p_w_prev}
\end{eqnarray}
for all $t\in [0,T_*]$ with
$C_M=\max\{C^{reg}_1,C^{reg}_2\}$. 
Now, for $\epsilon^*\in\mathbb{R}^+$ we define $\Upsilon_i$ for $i=1,\ldots,5$
as follows
\begin{eqnarray*}
&&\Upsilon_1=(\epsilon^*)^{-2}\Big((\epsilon^*)^2-\Upsilon_2\Big),
\quad
\Upsilon_2=2^{-1}\beta |\Omega|^{1/2}C_{gn}C_{poi}C_M\;\epsilon^*,
\quad
\Upsilon_3=(2\mu_r +C_3)C_M,
\\&&
\Upsilon_4=\beta |\Omega|^{1/2}(\Uppi_1+1)C_M,
\quad
\Upsilon_5=(\beta)^{1/2}C_M.
\end{eqnarray*}
Thus, selecting $\epsilon^*$ such that $(\epsilon^*)^2>\Upsilon_2$
and applying the Young inequality to the last two terms
of \eqref{eq:regularity_u_p_w_prev},
we get \eqref{eq:regularity_u_p_w}.
\end{proof}

\begin{lemma}
\label{lem:u_ut_w_wt_estimate}
There exists $T_1\in [0,T_*]$  and $\Uptheta:[0,T_1]\to\R^+$ independients of $f$ and $g$ such that
the following estimate holds 
\begin{eqnarray}
\|\bu(\cdot,t)\|_{\bH^1_0}+
\|\bw(\cdot,t)\|_{\bH^1_0}\le \Uptheta(t),
\label{eq:lem:u_ut_w_wt_estimate}
\end{eqnarray}
 for all $t\in [0,T_1].$
\end{lemma}

\begin{proof}
Testing the equations \eqref{eq:momento_lineal} and \eqref{eq:momento_angular} by $\bu_t$ 
and $\bw_t$, respectively;  
 summing the results; and applying the Minkowski and H\"older inequalities, we get
 \begin{eqnarray}
&&\frac{(\mu+\mu_r)}{2}\frac{d}{dt}\int_{\Omega}|\nabla \bu(\bx,t)|^2d\bx
+\frac{(c_0+2c_d)}{2}\frac{d}{dt}\int_{\Omega}|\nabla \bw(\bx,t)|^2d\bx
+\int_{\Omega}\Big(\rho |\bu_t|^2\Big)(\bx,t)d\bx
\nonumber\\&&
\qquad
+\int_{\Omega}\Big(\rho |\bw_t|^2\Big)(\bx,t)d\bx
=-\int_{\Omega}\Big(\rho (\bu\cdot\nabla )\bu\cdot\bu_t\Big)(\bx,t) d\bx
+2\mu_r\int_{\Omega}\Big(\curl \bw\,\cdot \bu_t\Big)(\bx,t) d\bx
\nonumber\\&&
\qquad
+\mu_r f(t)\int_{\Omega}\Big(\rho(\nabla h-\bm)\cdot\bu_t\Big)(\bx,t) d\bx
-\int_{\Omega}\Big(\rho (\bw\cdot\nabla )\bw\cdot\bw_t\Big)(\bx,t) d\bx
\nonumber\\&&
\qquad
+2\mu_r\int_{\Omega}\Big(\curl \bu\;\cdot \bw_t \Big)(\bx,t) d\bx+\mu_r g(t)
\int_{\Omega}\Big(\rho\bq\cdot\bw_t\Big)(\bx,t) d\bx
\nonumber\\&&
\qquad
\le
\Big[
\|(\rho \bu_t)(\cdot,t)\|_{\bL^2}\|\bu(\cdot,t)\|_{\bL^6}\|\nabla \bu(\cdot,t)\|_{\bL^3}
	+\|(\rho \bw_t)(\cdot,t)\|_{\bL^2}\|\bw(\cdot,t)\|_{\bL^6}
	\|\nabla \bw(\cdot,t)\|_{\bL^3}
\Big]
\nonumber\\&&
\qquad\quad
	+\Big[
	2\mu_r\Big\{
	\|\nabla \bw(\cdot,t)\|_{\bL^2} \|\bu_t (\cdot,t)\|_{\bL^2}
	+\|\nabla \bu(\cdot,t)\|_{\bL^2} \|\bw_t (\cdot,t)\|_{\bL^2}\Big\}
	\Big]
\nonumber\\&&
\qquad\quad
	+\Big[
	|f(t)|\|(\rho \bu_t)(\cdot,t)\|_{\bL^2}
	\Big(\|\nabla h(\cdot,t)\|_{\bL^2}+\|\bm(\cdot,t)\|_{\bL^2}\Big)
	+
	|g(t)|\|(\rho \bw_t)(\cdot,t)\|_{\bL^2}\|\bq(\cdot,t)\|_{\bL^2}
	\Big]
\nonumber\\&&
\qquad\quad
	+\Big[
	4\mu_r\|\bw(\cdot,t)\|_{\bL^2}\|\bw_t(\cdot,t)\|_{\bL^2}
	\Big]
	:=\sum_{i=1}^4J_i,
\label{eq:ml_and_ma_times_ut_wt}
\end{eqnarray}
where each $J_i$ are defined by the corresponding brackets $\big[\quad\big]$.
Now, we will prove the estimate by getting  
some bounds for each $J_i$ and then applying a Gronwall type inequality.
Indeed, first, for $J_1$, by Lemma~\ref{lema:aprioriestimatesforrho}, 
inequality \eqref{eq:spoingag}, Young inequality  and 
Lemma~\ref{prop:auxiliar_stokes_elliptic_results}, we deduce that
\begin{eqnarray}
&&J_1
\le
\sqrt{\beta} C_{poi} C_{gn}
\Big[
\|(\sqrt{\rho} \bu_t)(\cdot,t)\|_{\bL^2}\|\nabla \bu(\cdot,t)\|^{3/2}_{\bL^2}
\| \bu(\cdot,t)\|^{1/2}_{\bH^2}
+\|(\sqrt{\rho} \bw_t)(\cdot,t)\|_{\bL^2}\times
\nonumber
\\&&
\qquad\quad
\|\nabla \bw(\cdot,t)\|^{3/2}_{\bL^2}\| \bw(\cdot,t)\|^{1/2}_{\bH^2}\Big]
\nonumber
\\&&
\le 
\frac{\sqrt{\beta} C_{poi} C_{gn}}{2}\Big\{
\|(\sqrt{\rho} \bu_t)(\cdot,t)\|_{\bL^2}\;
\Big[\|\nabla \bu(\cdot,t)\|^{3}_{\bL^2}+\| \bu(\cdot,t)\|_{\bH^2}\Big]
+
\|(\sqrt{\rho} \bw_t)(\cdot,t)\|_{\bL^2}\times
\nonumber
\\&&
\qquad
\Big[\|\nabla  \bw(\cdot,t)\|^{3}_{\bL^2}+\| \bw(\cdot,t)\|_{\bH^2}\Big]
\Big\}
\nonumber
\\&&
\le
\left(\frac{\Uppsi_1}{\epsilon^u}+\Uppsi_2\right)
\|(\sqrt{\rho} \bu_t)(\cdot,t)\|^{2}_{\bL^2}
+\left(\frac{\Uppsi_3}{\epsilon^w}+\Uppsi_2\right)
\|(\sqrt{\rho} \bw_t)(\cdot,t)\|^{2}_{\bL^2}
+\Uppsi_4
\Big[\|\nabla \bu(\cdot,t)\|^{6}_{\bL^2}
\nonumber
\\&&
\qquad
+\|\nabla \bw(\cdot,t)\|^{6}_{\bL^2}\Big]
+\Uppsi_5
\Big[\|\nabla \bu(\cdot,t)\|^{2}_{\bL^2}
+\|\nabla \bw(\cdot,t)\|^{2}_{\bL^2}\Big]
+\Uppsi_6|f(t)|^{2}+\Uppsi_7|g(t)|^{2}
\label{eq:inequality_J1}
\end{eqnarray}
where
\begin{eqnarray*}
&&\Uppsi_1=\frac{\sqrt{\beta} C_{poi} C_{gn}}{4}
	\Big(1+2\frac{\Upsilon_2}{\Upsilon_1}+2\frac{\Upsilon_3}{\Upsilon_1}
	+2\frac{\Upsilon_4}{\Upsilon_1}\|\bm(\cdot,t)\|_{\bL^2}\Big),
\quad
\Uppsi_2=\sqrt{\beta} C_{poi} C_{gn}\frac{\Upsilon_5}{\Upsilon_1},
\\&&
\Uppsi_3=\frac{\sqrt{\beta} C_{poi} C_{gn}}{4}
	\Big(1+2\frac{\Upsilon_2}{\Upsilon_1}+2\frac{\Upsilon_3}{\Upsilon_1}
	+2\frac{\Upsilon_4}{\Upsilon_1}\|\bq(\cdot,t)\|_{\bL^2}\Big),
\\&&
\Uppsi_4=\frac{\sqrt{\beta} C_{poi} C_{gn}}{4}
	\Big(1+2\frac{\Upsilon_2}{\Upsilon_1}\Big)
	\max\{\epsilon^\bu,\epsilon^\bw\},
	\quad
	\Uppsi_5=\frac{\sqrt{\beta} C_{poi} C_{gn}}{2}\frac{\Upsilon_3}{\Upsilon_1},
\\&&
\Uppsi_6=\frac{\sqrt{\beta} C_{poi} C_{gn}}{2}\frac{\Upsilon_4}{\Upsilon_1}
	\|\bm(\cdot,t)\|_{\bL^2}
\;\mbox{ and }\;
\Uppsi_7=\frac{\sqrt{\beta} C_{poi} C_{gn}}{2}\frac{\Upsilon_4}{\Upsilon_1}
	\|\bq(\cdot,t)\|_{\bL^2}.
\end{eqnarray*}
Now, for $J_i$, $i=2,3,4$, by inequality \eqref{eq:spoingag},
Lemmas~\ref{lema:aprioriestimatesforrho}-\ref{lema:aprioriestimatesforhandr} 
and Young inequality, we deduce that
\begin{eqnarray}
&&J_2\le
\frac{\mu_r}{\alpha}\Big[\frac{1}{\epsilon^\bu}
\|(\sqrt{\rho} \bu_t)(\cdot,t)\|^{2}_{\bL^2}
	+\frac{1}{\epsilon^\bw}
	\|(\sqrt{\rho} \bw_t)(\cdot,t)\|^{2}_{\bL^2}\Big]
\nonumber
\\&&
\qquad
	+\mu_r\max\{\epsilon^\bu,\epsilon^\bw\}
	\Big[\|\nabla \bu(\cdot,t)\|^{2}_{\bL^2}
	+\|\nabla \bw(\cdot,t)\|^{2}_{\bL^2}\Big],
\label{eq:inequality_J2}
\\&&
J_3\le
\frac{\sqrt{\beta}(\|\nabla h(\cdot,t)\|_{\bL^2}
	+\|\bm(\cdot,t)\|_{\bL^2})}{2\;\epsilon^\bu}
	\|(\sqrt{\rho} \bu_t)(\cdot,t)\|^{2}_{\bL^2}
	+
	\frac{\sqrt{\beta}\|\bq(\cdot,t)\|_{\bL^2}}{2\;\epsilon^\bw}
	\|(\sqrt{\rho} \bw_t)(\cdot,t)\|^{2}_{\bL^2}
\nonumber
\\&&
\qquad
+\frac{\sqrt{\beta}}{2}\Big(\|\nabla h(\cdot,t)\|_{\bL^2}
	+\|\bm(\cdot,t)\|_{\bL^2}\Big)\;
	\epsilon^\bu\;|f(t)|^{2}
	+
	\frac{\sqrt{\beta}}{2}\|\bq(\cdot,t)\|_{\bL^2}
	\;\epsilon^\bw\;|g(t)|^{2}
\nonumber
\\&&
\hspace{0.5cm}
\le
\frac{\sqrt{\beta}}{2}(\Uppi_1+1)
\Big\{\frac{1}{\epsilon^\bu}
	\|\bm(\cdot,t)\|_{\bL^2}\|(\sqrt{\rho} \bu_t)(\cdot,t)\|^{2}_{\bL^2}
	+\frac{1}{\epsilon^\bw}
	\|\bq(\cdot,t)\|_{\bL^2} \|(\sqrt{\rho} \bw_t)(\cdot,t)\|^{2}_{\bL^2}\Big\}
\nonumber
\\&&
\qquad
+\frac{\sqrt{\beta}}{2}(\Uppi_1+1)\Big\{\epsilon^\bu
\|\bm(\cdot,t)\|_{\bL^2}|f(t)|^{2}
+\epsilon^\bw\|\bq(\cdot,t)\|_{\bL^2}|g(t)|^{2}\Big\},
\label{eq:inequality_J3}
\\&&
J_4\le
\frac{2\mu_r C_{poi}\;\epsilon^\bw\;}{\alpha}
\|\nabla \bw(\cdot,t)\|^{2}_{\bL^2}
+\frac{2\mu_r C_{poi}}{\alpha\;\epsilon^\bw\;}
\|(\sqrt{\rho} \bw_t)(\cdot,t)\|^{2}_{\bL^2}.
\label{eq:inequality_J4}
\end{eqnarray}
Inserting \eqref{eq:inequality_J1}-\eqref{eq:inequality_J4} in 
\eqref{eq:ml_and_ma_times_ut_wt} and selecting 
$\epsilon^{\ell}>\max\{N^{\ell}(1-\Uppsi_2)^{-1},0\}$ for $\ell\in\{\bu,\bw\}$, where
\begin{eqnarray*}
N^\bu &=&\Uppsi_1+\frac{\mu_r}{\alpha}+\frac{(\beta)^{1/2}}{2}
(\Uppi_1+1)\|\bm(\cdot,t)\|_{L^2}\quad\mbox{and}
\\
N^\bw &=&\Uppsi_2+\frac{\mu_r}{\alpha}(2C_{poi}+1)+\frac{(\beta)^{1/2}}{2}
(\Uppi_1+1)\|\bq(\cdot,t)\|_{L^2},
\end{eqnarray*}
we find that there exists $\Upxi_1,\Upxi_2$ and $\Upxi_3$ defined
as follows
\begin{eqnarray*}
 \Upxi_1&=&\Uppsi_4\mathcal{C}^{-1},
 \qquad
 \Upxi_2=\mathcal{C}^{-1}\left(\Uppsi_5+\mu_r\epsilon^\bw\alpha^{-1}
(2C_{poi}+1)\right),
\\
\Upxi_3&=&\mathcal{C}^{-1}
\max\Big\{\Uppsi_6+\epsilon^\bu\|\bm(\cdot,t)\|_{L^2},\;\Uppsi_7
+\epsilon^\bw\|\bq(\cdot,t)\|_{L^2}\Big\}
\end{eqnarray*}
with 
$
\mathcal{C}=\min\{2^{-1}(\mu+\mu_r),2^{-1}(c_0+2c_d),\;
1-\Uppsi_2-(\epsilon^\bu)^{-1}N^\bu,\;1-\Uppsi_2-(\epsilon^\bw)^{-1}N^\bw \}
$ and
such that the inequality
\begin{eqnarray}
&&\frac{d}{dt}\Big(\|\nabla \bu(\cdot,t)\|^{2}_{\bL^2}
+\|\nabla \bw(\cdot,t)\|^{2}_{\bL^2}\Big)
+\|(\sqrt{\rho} \bv_t)(\cdot,t)\|^{2}_{\bL^2}
+\|(\sqrt{\rho} \bw_t)(\cdot,t)\|^{2}_{\bL^2}
\le\Upxi_1
\Big[\|\nabla \bu(\cdot,t)\|^{6}_{\bL^2}
\nonumber\\&&
\qquad
+\|\nabla \bw(\cdot,t)\|^{6}_{\bL^2}\Big]
+\Upxi_2
\Big[\|\nabla \bu(\cdot,t)\|^{2}_{\bL^2}
+\|\nabla \bw(\cdot,t)\|^{2}_{\bL^2}\Big]
+\Upxi_3\Big[|f(t)|^{2}+|g(t)|^{2}\Big]
\label{lem:u_ut_w_wt_main_estimate}
\end{eqnarray}
holds for $t\in [0,T_{*}]$. Now,
making use of  Lemma~3 given on \cite{heywood_1980}, we conclude the
existence of $T_1$ depending on $\|\nabla \bu(\cdot,0)\|_{\bL^2}$ and
$\|\nabla \bw(\cdot,0)\|_{\bL^2}$ such  the estimate
\eqref{eq:lem:u_ut_w_wt_estimate} holds
with $\Uptheta$ depending only on 
$\Upxi_1,\Upxi_2,\Upxi_3,\|\nabla \bu(\cdot,0)\|_{\bL^2}$ and 
$\|\nabla \bw(\cdot,0)\|_{\bL^2}$.
\end{proof}

\begin{lemma}
\label{lem:estimates3}
Consider $T_1$ as is given on Lemma~\ref{lem:u_ut_w_wt_estimate}. Then,
there exists $\Upphi_i$, $i=1,\ldots,6,$ independent of $f$ and $g$ such that
the following estimate holds 
\begin{eqnarray}
&&\frac{d}{dt}\|\sqrt{\rho} \bu_t(\cdot,t)\|^2_{\bL^2}
+\| \nabla \bu_t(\cdot,t)\|^2_{\bL^2}
+\frac{d}{dt}\|\sqrt{\rho} \bw_t(\cdot,t)\|^2_{\bL^2}
+\| \nabla \bw_t(\cdot,t)\|^2_{\bL^2} 
\nonumber
\\&&
\quad
\le 
\Upphi_1
\Big[|f'(t)|
\|\sqrt{\rho} \bu_t(\cdot,t)\|_{\bL^2}
+|g'(t)|\|\sqrt{\rho} \bw_t(\cdot,t)\|_{\bL^2}\Big]
+\Upphi_2\Big[
|f(t)|\Big\{\|\nabla h_t(\cdot,t)\|_{\bL^2}+1\Big\}
\times
\nonumber
\\&&
\qquad\quad
\|\sqrt{\rho} \bu_t(\cdot,t)\|_{\bL^2}
+|g(t)|
\|\sqrt{\rho} \bw_t(\cdot,t)\|_{\bL^2}
\Big]
+
\Upphi_3\|\nabla\rho(\cdot,t)\|^2_{\bL^q}\Big[|f(t)|^2
\Big(\|h(\cdot,t)\|^2_{H^2}+1\Big)
\nonumber
\\&&
\qquad
+|g(t)|^2\|\nabla\rho(\cdot,t)\|^2_{\bL^q}
\Big]
+\Upphi_4
\|\nabla\rho(\cdot,t)\|^{4q/(3q-6)}_{\bL^q}
\Big[\|\sqrt{\rho}\bu_t(\cdot,t)\|^{2}_{\bH^2}
+\|\sqrt{\rho}\bw_t(\cdot,t)\|^{2}_{\bH^2}\Big]
\nonumber
\\&&
\qquad
+\Upphi_5
\|\nabla\rho(\cdot,t)\|^2_{L^q}\Big[\|\bu(\cdot,t)\|^{6/q}_{\bH^2}
+\|\bw(\cdot,t)\|^{6/q}_{\bH^2}\Big]
\nonumber
\\&&
\qquad
+
\Upphi_6\Big[\|\sqrt{\rho} \bu_t(\cdot,t)\|^{2}_{\bL^2}
+\|\sqrt{\rho} \bw_t(\cdot,t)\|^{2}_{\bL^2}\Big],
\label{eq:lem:estimates3}
\end{eqnarray}
for all $t\in [0,T_1].$
\end{lemma} 
\begin{proof}
Diferentiating \eqref{eq:momento_lineal} 
and  \eqref{eq:momento_angular} with respect to $t$; 
testing the results by $\bu_t$ and  $\bw_t$,
respectively;
summing the resulting equations; and rearranging the terms we get
\begin{eqnarray}
&&\frac{1}{2}\frac{d}{dt}\int_{\Omega}\rho|\bu_t(\bx,t)|^2d\bx
+(\mu+\mu_r)\int_{\Omega}|\nabla u_t(\bx,t)|^2d\bx
+\frac{1}{2}\frac{d}{dt}\int_{\Omega}\rho| \bw_t(\bx,t)|^2d\bx
\nonumber\\&&
\quad
+(c_a+c_d)\int_{\Omega}|\nabla \bw_t(\bx,t)|^2d\bx
=2\mu_r\left[\int_{\Omega}\curl \cdot \bw(x,t)\bu_t(\bx,t)d\bx
+\int_{\Omega}\curl \bu(x,t)\cdot\bw_t(\bx,t)d\bx\right]
\nonumber\\&&
\quad
+\left[\int_{\Omega}f'(t)\Big(\rho(\nabla h-\bm)\cdot\bu_t\Big)(\bx,t)d\bx
+\int_{\Omega}g'(t)\Big(\rho \bq\cdot\bw_t\Big)(\bx,t)d\bx\right]
\nonumber\\&&
\quad
+\left[\int_{\Omega}f(t)\Big(\rho(\nabla h_t-\bm_t)\cdot\bu_t\Big)(\bx,t)d\bx
+\int_{\Omega}g(t)\Big(\rho\bq_t\cdot\bw_t\Big)(\bx,t)d\bx\right]
\nonumber\\&&
\quad
+\left[
\int_{\Omega}f(t)\Big(\rho_t(\nabla h-\bm)\cdot\bu_t\Big)(\bx,t)d\bx
+\int_{\Omega}g(t)\Big(\rho_t\bq \cdot\bw_t\Big)(\bx,t)d\bx
\right]
\nonumber\\&&
\quad
-\frac{1}{2}\left[\int_{\Omega}\Big(\rho_t |\bu_t|^2\Big)(\bx,t)d\bx
+\int_{\Omega}\Big(\rho_t |\bw_t|^2\Big)(\bx,t)d\bx\right]
\nonumber\\&&
\quad
+\left[
\int_{\Omega}\Big(\rho_t (\bu\cdot\nabla )\bu\cdot \bu_t\Big)(\bx,t)d\bx
+\int_{\Omega}\Big(\rho_t (\bw\cdot\nabla )\bw \cdot\bw_t\Big)(\bx,t)d\bx
\right]
\nonumber\\&&
\quad
-\left[
\int_{\Omega}\Big(\rho (\bu_t\cdot\nabla )\bu\cdot \bu_t\Big)(\bx,t)d\bx
+\int_{\Omega}\Big(\rho (\bw_t\cdot\nabla )\bw\cdot \bw_t\Big)(\bx,t)d\bx
\right]
= \sum_{i=0}^6 I_i,
\label{eq:lem:estimates3_full}
\end{eqnarray}
where $I_i$ for 
$i=0,\ldots,6$ are defined by the brackets.
Hence, the proof of \eqref{eq:lem:estimates3} is reduced to
get some bounds for 
each $I_i$ based on Minkowski and H\"older inequalities
and the previous Lemmas as will be specified below.
First, by applying the  
Lemmas~\ref{lema:aprioriestimatesforrho},\ref{lem:u_ut_w_wt_estimate}
and Young inequality, we find that $I_0$ can be bounded as follows
\begin{eqnarray}
I_0
&\le& 2\mu_r\left|\int_{\Omega}\Big(\curl \bw\cdot\bu_t\Big)(\bx,t)d\bx
+\int_{\Omega}\Big(\curl \bu\cdot\bw_t\Big)(\bx,t)d\bx\right|
\nonumber
\\
\quad
&\le& 2\mu_r\Big(\|\nabla \bw(\cdot,t)\|_{\bL^2}\|\bu_t(\cdot,t)\|_{\bL^2}
+\|\nabla \bu(\cdot,t)\|_{\bL^2}\|\bw_t(\cdot,t)\|_{\bL^2}\Big)
\nonumber
\\
\quad
&\le &(\alpha)^{-1/2}\mu_r\Uptheta(t)\Big(
\|\sqrt{\rho}\bu_t(\cdot,t)\|^2_{\bL^2}+\|\sqrt{\rho}\bw_t(\cdot,t)\|^2_{\bL^2}
\Big),
\label{eq:lem:estimates3_I0}
\end{eqnarray}
for all $t\in [0,T_1].$
Now, by Lemmas~\ref{lema:aprioriestimatesforrho} and 
\ref{lema:aprioriestimatesforhandr}, we get that
\begin{eqnarray}
I_1
&\le&\left|\int_{\Omega}f'(t)\Big(\rho(\nabla h-\bm)\cdot\bu_t\Big)(\bx,t)d\bx
    +\int_{\Omega}g'(t)\Big(\rho \bq\cdot\bw_t\Big)(\bx,t)d\bx\right|
\nonumber
\\
\quad
&\le& \sqrt{\beta}\Big(
|f'(t)|\|(\nabla h-\bm)(\cdot,t)\|_{\bL^2}
\|\sqrt{\rho} \bu_t(\cdot,t)\|_{\bL^2}
+|g'(t)|\|\bq(\cdot,t)\|_{\bL^2}\|\sqrt{\rho} \bw_t(\cdot,t)\|_{\bL^2}\Big)
\nonumber
\\
\quad
&\le &\sqrt{\beta}
\Big(\Uppi_1+1\Big)\;
\Big(|f'(t)|\|\bm(\cdot,t)\|_{\bL^2}\|\sqrt{\rho} \bu_t(\cdot,t)\|_{\bL^2}
+|g'(t)|\|\bq(\cdot,t)\|_{\bL^2}\|\sqrt{\rho} \bw_t(\cdot,t)\|_{\bL^2}\Big)
\nonumber
\\
\quad
&\le &\overline{\Upphi}_1
\Big(|f'(t)|\|\sqrt{\rho} \bu_t(\cdot,t)\|_{\bL^2}
+|g'(t)|\|\sqrt{\rho} \bw_t(\cdot,t)\|_{\bL^2}\Big),
\label{eq:lem:estimates3_I1}
\end{eqnarray}
where
$
\overline{\Upphi}_1=\sqrt{\beta} \Big(\Uppi_1+1\Big)
\max\{\|\bm(\cdot,t)\|_{\bL^2},\|\bq(\cdot,t)\|_{\bL^2}\}.
$
In the case of $I_2$, by applying Lemma~\ref{lema:aprioriestimatesforrho}, 
we have that
\begin{eqnarray}
I_2
&\le &\left|\int_{\Omega}\Big(f\rho(\nabla h_t-\bm_t)\cdot\bu_t\Big)(\bx,t)d\bx
    +\int_{\Omega}\Big(g\rho \bq_t\cdot\bw_t\Big)(\bx,t)d\bx\right|
\nonumber
\\
\quad
&\le& \sqrt{\beta}\Big(
|f(t)|
\Big\{\|\nabla h_t(\cdot,t)\|_{\bL^2}+\|\bm_t(\cdot,t)\|_{\bL^2}\Big\}
\|\sqrt{\rho} \bu_t(\cdot,t)\|_{\bL^2}
\nonumber
\\&&
\hspace{1.5cm}
+|g(t)|\|\bq_t(\cdot,t)\|_{\bL^2}\|\sqrt{\rho} \bw_t(\cdot,t)\|_{\bL^2}\Big),
\nonumber
\\
\quad
&\le& \overline{\Upphi}_2\Big(
|f(t)|
\Big\{\|\nabla h_t(\cdot,t)\|_{\bL^2}+1\Big\}
\|\sqrt{\rho} \bu_t(\cdot,t)\|_{\bL^2}
\nonumber
\\&&
\hspace{1.5cm}
+|g(t)|\Big\{\|\nabla r_t(\cdot,t)\|_{\bL^2}+1\Big\}
\|\sqrt{\rho} \bw_t(\cdot,t)\|_{\bL^2}
\Big),
\label{eq:lem:estimates3_I2}
\end{eqnarray}
where
$
\overline{\Upphi}_2=\sqrt{\beta}\;
\max\{\|\bm_t(\cdot,t)\|_{\bL^2},\;
\|\bq_t(\cdot,t)\|_{\bL^2},\;1\}.
$
For $I_3$, by  equation \eqref{eq:ecuacion_continuidad},
inequality \eqref{eq:spoingag},
Lemmas~\ref{lema:aprioriestimatesforrho} and
\ref{lem:u_ut_w_wt_estimate} and noticing  that
\begin{eqnarray*}
\|(\nabla h-\bm)(\cdot,t)\|_{\bL^3}&\le 2^{-1}
\Big(C^2_{gn}\Uppi^2_1\|\bm(\cdot,t)\|_{\bL^2}\|h(\cdot,t)\|_{H^2}
+\|\bm(\cdot,t)\|^2_{\bL^3}\Big)
\end{eqnarray*}
we deduce that
\begin{eqnarray}
I_3
&\le&\left|\int_{\Omega}f(t)\Big(\rho_t(\nabla h-\bm)\cdot\bu_t\Big)(\bx,t)d\bx
    +\int_{\Omega}g(t)\Big(\rho_t\bq\cdot\bw_t\Big)(\bx,t)d\bx\right|
\nonumber
\\
\quad
&=&\left|\int_{\Omega}f(t)\Big((\bu\cdot\nabla\rho)(\nabla h-\bm)\cdot\bu_t\Big)(\bx,t)d\bx
    +\int_{\Omega}g(t)\Big((\bw\cdot\nabla\rho)\bq\cdot\bw_t\Big)(\bx,t)d\bx\right|
\nonumber
\\
\quad
&\le& 
|f(t)|
\|\bu(\cdot,t)\|_{\bL^6}\|\nabla\rho(\cdot,t)\|_{\bL^3}
\|(\nabla h-\bm)(\cdot,t)\|_{\bL^3}\|\bu_t(\cdot,t)\|_{\bL^6}
\nonumber
\\&&
\qquad
+|g(t)|\|\bw(\cdot,t)\|_{\bL^6}\|\nabla\rho(\cdot,t)\|_{\bL^3}
\|\bq(\cdot,t)\|_{\bL^3}\|\bw_t(\cdot,t)\|_{\bL^6}
\nonumber
\\
\quad
&\le&
C_{poi}\Uptheta(t)\Big\{
|f(t)|\|\nabla\rho(\cdot,t)\|_{\bL^3}
\|(\nabla h-\bm)(\cdot,t)\|_{\bL^3}\|u_t(\cdot,t)\|_{\bL^6}
\nonumber
\\&&
\hspace{2.1cm}
+|g(t)|\|\nabla\rho(\cdot,t)\|_{\bL^3}
\|\bq(\cdot,t)\|_{\bL^3}\|\bw_t(\cdot,t)\|_{\bL^6}\Big\}
\nonumber
\\
\quad
&\le&
2^{-1}C^2_{poi}\Uptheta(t) 
\Big\{
|f(t)|\|\nabla\rho(\cdot,t)\|_{\bL^3}
\Big(
C^2_{gn}\Uppi^2_1\|\bm(\cdot,t)\|_{\bL^2}\|h(\cdot,t)\|_{H^2}
+\|\bm(\cdot,t)\|^2_{\bL^3}
\Big)
\times
\nonumber
\\&&
\hspace{2.6cm}
\|\nabla \bu_t(\cdot,t)\|_{\bL^2}
+|g(t)|\|\nabla\rho(\cdot,t)\|_{\bL^3}
\|\bq(\cdot,t)\|^2_{L^3}\|\nabla \bw_t(\cdot,t)\|_{\bL^2}\Big\}
\nonumber
\\
\quad
&\le&
\frac{1}{2\epsilon_\bu}\|\nabla \bu_t(\cdot,t)\|^2_{\bL^2}
+\overline{\Upphi}^\bu_3|f(t)|^2\|\nabla\rho(\cdot,t)\|^2_{\bL^q}
\Big(\|h(\cdot,t)\|^2_{H^2}+1\Big)
\nonumber
\\&&
\quad\;\; +
\frac{1}{2\epsilon_\bw}\|\nabla \bw_t(\cdot,t)\|^2_{\bL^2}
+\overline{\Upphi}^\bw_3|g(t)|^2\|\nabla\rho(\cdot,t)\|^2_{\bL^q},
\label{eq:lem:estimates3_I3}
\end{eqnarray}
for all $t\in[0,T_1]$,
where $\overline{\Upphi}^{\ell}_3=3\epsilon_{\ell}
|\Omega|^{2(q-3)/3q}C^4_{poi}\Uptheta(t)^2\mathcal{L}^2/8$
with $\ell\in\{\bu,\bw\}$ and $\mathcal{L}$ is defined as follows
\begin{eqnarray*}
\mathcal{L}=
\max\Big\{2^{-1}C^2_{gn}\Uppi^2_1\|\bm(\cdot,t)\|^2_{\bL^2}+
\|\bm(\cdot,t)\|^2_{\bL^3},
2^{-1}C^2_{gn}\Uppi^2_1\|\bq(\cdot,t)\|^2_{\bL^2}+
\|\bq(\cdot,t)\|^2_{\bL^3}\}.
\end{eqnarray*}
The term $I_4$  can be bounded by the application of
equation \eqref{eq:ecuacion_continuidad}, the inequality \eqref{eq:spoingag}
and Lemma~\ref{lem:u_ut_w_wt_estimate}, since
we can perform the following calculus
\begin{eqnarray}
I_4
&\le& \frac{1}{2}\left|\int_{\Omega}\Big(\rho_t |\bu_t|^2\Big)(\bx,t)d\bx
    +\int_{\Omega}\Big(\rho_t |\bw_t|^2\Big)(\bx,t)d\bx\right|
\nonumber
\\
\quad
&=&\frac{1}{2}\left|\int_{\Omega}\Big((\bu\cdot\nabla\rho) |\bu_t|^2\Big)(\bx,t)d\bx
    +\int_{\Omega}\Big((\bw\cdot\nabla\rho) |\bw_t|^2\Big)(\bx,t)d\bx\right|
\nonumber
\\
\quad
&\le&
\|\bu(\cdot,t)\|_{\bL^6}\|\nabla\rho(\cdot,t)\|_{\bL^q}
\|\bu_t(\cdot,t)\|_{\bL^{2q/(q-2)}}
\|\bu_t(\cdot,t)\|_{\bL^3}
\nonumber
\\&&
\qquad
+\|\bw(\cdot,t)\|_{\bL^6}\|\nabla\rho(\cdot,t)\|_{\bL^q}
\|\bw_t(\cdot,t)\|_{\bL^{2q/(q-2)}}
\|\bw_t(\cdot,t)\|_{\bL^3}
\nonumber
\\
\quad
&\le&
C_{poi}\Uptheta(t)\Big\{
\|\nabla\rho(\cdot,t)\|_{\bL^q}\|\bu_t(\cdot,t)\|_{\bL^{2q/(q-2)}}
\|\bu_t(\cdot,t)\|_{\bL^3}
\nonumber
\\&&
\quad
+\|\nabla\rho(\cdot,t)\|_{\bL^q}\|\bw_t(\cdot,t)\|_{\bL^{2q/(q-2)}}
\|\bw_t(\cdot,t)\|_{\bL^3}
\Big\}
\nonumber
\\
\quad
&\le&
C_{poi}C^2_{gn}\Uptheta(t) \Big\{
\|\nabla\rho(\cdot,t)\|_{\bL^q}
\|\bu_t(\cdot,t)\|^{1-3/q}_{\bL^2}\|\nabla \bu_t(\cdot,t)\|^{3/q}_{\bL^2}
\|\bu_t(\cdot,t)\|^{1/2}_{\bL^2}\|\nabla \bu_t(\cdot,t)\|^{1/2}_{\bL^2}
\nonumber
\\&&
\qquad
+\|\nabla\rho(\cdot,t)\|_{\bL^q}
\|\bw_t(\cdot,t)\|^{1-3/q}_{\bL^2}\|\nabla \bw_t(\cdot,t)\|^{3/q}_{\bL^2}
\|\bw_t(\cdot,t)\|^{1/2}_{\bL^2}\|\nabla \bw_t(\cdot,t)\|^{1/2}_{\bL^2}
\Big\}
\nonumber
\\
\quad&\le&
\frac{1}{2\epsilon_\bu}\|\nabla \bu_t(\cdot,t)\|^{2}_{\bL^2}
+\overline{\Upphi}^\bu_4
\|\nabla\rho(\cdot,t)\|^{4q/(3q-6)}_{\bL^q}\|\sqrt{\rho}\bu_t(\cdot,t)\|^{2}_{\bH^2}
\nonumber
\\&&
\qquad
+\frac{1}{2\epsilon_\bw}\|\nabla \bw_t(\cdot,t)\|^{2}_{\bL^2}
+\overline{\Upphi}^\bw_4
\|\nabla\rho(\cdot,t)\|^{4q/(q+6)}_{\bL^q}\|\sqrt{\rho}\bw_t(\cdot,t)\|^{2}_{\bH^2},
\label{eq:lem:estimates3_I4}
\end{eqnarray}
where $\overline{\Upphi}^{\ell}_4=(\alpha)^{-1}\Big((2\epsilon_{\ell})^{(q+6)}
(C_{poi}C^2_{gn}\Uptheta(t))^{4q}\Big)^{1/(3q-6)}$
with $\ell\in\{\bu,\bw\}$.
An application of
equation \eqref{eq:ecuacion_continuidad}, inequality
\eqref{eq:spoingag}
and Lemma~\ref{lem:u_ut_w_wt_estimate} implies 
the following bound for~$I_5$
\begin{eqnarray}
I_5
&\le& \left|\int_{\Omega}\Big(\rho_t (\bu\cdot\nabla)\bu\cdot \bu_t\Big)(\bx,t)d\bx
    +\int_{\Omega}\Big(\bw\cdot\nabla)\bw \cdot\bw_t\Big)(\bx,t)d\bx\right|
\nonumber
\\
\quad
&=&\frac{1}{2}\left|\int_{\Omega}\Big((\bu\cdot\nabla\rho) \bu_t\cdot\bu_t\Big)(\bx,t)d\bx
    +\int_{\Omega}\Big((\bw\cdot\nabla\rho) \bw_t\cdot\bw_t\Big)(\bx,t)d\bx\right|
\nonumber
\\
\quad
&\le&
\|\bu(\cdot,t)\|^2_{\bL^6}\|\bu_t(\cdot,t)\|_{\bL^6}\|\nabla\rho(\cdot,t)\|_{\bL^q}
\|\nabla \bu(\cdot,t)\|_{\bL^{2q/(q-2)}}
\nonumber
\\&&
\qquad
+\|\bw(\cdot,t)\|^2_{L^6}\|\bw_t(\cdot,t)\|_{\bL^6}\|\nabla\rho(\cdot,t)\|_{\bL^q}
\|\nabla \bw(\cdot,t)\|_{\bL^{2q/(q-2)}}
\nonumber
\\
\quad
&\le&
C^3_{poi}\Uptheta(t)^2 C_{gn}\Big\{
\|\nabla \bu_t(\cdot,t)\|_{\bL^2}\|\nabla\rho(\cdot,t)\|_{\bL^q}
\|\nabla \bu(\cdot,t)\|^{1-3/q}_{\bL^2}\|\bu(\cdot,t)\|^{3/q}_{\bH^2}
\nonumber
\\&&
\qquad
+\|\nabla \bw_t(\cdot,t)\|_{\bL^2}\|\nabla\rho(\cdot,t)\|_{\bL^q}
\|\nabla \bw(\cdot,t)\|^{1-3/q}_{\bL^2}\|\bw(\cdot,t)\|^{3/q}_{\bH^2}
\Big\}
\nonumber
\\
\quad&\le&
\frac{1}{2\epsilon_\bu}\|\nabla \bu_t(\cdot,t)\|^{2}_{\bL^2}
+\overline{\Upphi}^u_5
\|\nabla\rho(\cdot,t)\|^2_{\bL^q}\|\bu(\cdot,t)\|^{6/q}_{\bH^2}
\nonumber
\\&&
\qquad+
\frac{1}{2\epsilon_\bw}\|\nabla \bw_t(\cdot,t)\|^{2}_{\bL^2}
+\overline{\Upphi}^\bw_5
\|\nabla\rho(\cdot,t)\|^2_{\bL^q}\|\bw(\cdot,t)\|^{6/q}_{\bH^2},
\label{eq:lem:estimates3_I5}
\end{eqnarray}
where $\overline{\Upphi}^{\ell}_5=2^{-1}\epsilon_{\ell}\;
C^6_{poi}\Uptheta(t)^{6(q-1)/q} C^2_{gn}$
with $\ell\in\{\bu,\bw\}$.
By inequality \eqref{eq:spoingag}
and Lemma~\ref{lem:u_ut_w_wt_estimate} we deduce that
\begin{eqnarray}
I_6
&\le&\left|\int_{\Omega}\Big(\rho (\bu_t\cdot\nabla)\bu\cdot \bu_t\Big)(\bx,t)d\bx
    +\int_{\Omega}\Big(\rho (\bw_t\cdot\nabla)\bw \cdot\bw_t\Big)dx\right|
\nonumber
\\
\quad
&\le&
\|\rho(\cdot,t)\|_{L^\infty}\|\nabla \bu(\cdot,t)\|_{\bL^2}
\| \bu_t(\cdot,t)\|_{L^3}\| \bu_t(\cdot,t)\|_{\bL^6}
\nonumber
\\&&
\qquad
+\|\rho(\cdot,t)\|_{L^\infty}\|\nabla \bw(\cdot,t)\|_{\bL^2}
\| \bw_t(\cdot,t)\|_{\bL^3}\| \bw_t(\cdot,t)\|_{\bL^6}
\nonumber
\\
\quad
&\le&
\beta C_{poi}C_{gn}\Uptheta(t) \Big\{
\| \bu_t(\cdot,t)\|^{1/2}_{\bL^2}\| \nabla \bu_t(\cdot,t)\|^{3/2}_{\bL^2}
+\| \bw_t(\cdot,t)\|^{1/2}_{\bL^2}\| \nabla \bw_t(\cdot,t)\|^{3/2}_{\bL^2}
\Big\}
\nonumber
\\
\quad&\le&
\frac{1}{2\epsilon_\bu }\|\nabla \bu_t(\cdot,t)\|^{2}_{\bL^2}
+\overline{\Upphi}^\bu_6\|\sqrt{\rho} \bu_t(\cdot,t)\|^{2}_{\bL^2}
+\frac{1}{2\epsilon_\bw }\|\nabla \bw_t(\cdot,t)\|^{2}_{\bL^2}
+\overline{\Upphi}^\bw_6
\|\sqrt{\rho} \bw_t(\cdot,t)\|^{2}_{\bL^2},
\qquad
\label{eq:lem:estimates3_I6}
\end{eqnarray}
where $\overline{\Upphi}^{\ell}_6=
(\alpha)^{-1}(2\epsilon_{\ell})^3(\beta C_{poi}C_{gn}\Uptheta(t))^4$
with $\ell\in\{\bu,\bw\}$.
Inserting \eqref{eq:lem:estimates3_I1}-\eqref{eq:lem:estimates3_I6} in
\eqref{eq:lem:estimates3_full} and selecting $\epsilon_\bu=2(\mu+\mu_r)^{-1}$
and $\epsilon_\bw=2(c_a+c_d)^{-1}$,
we deduce that \eqref{eq:lem:estimates3} holds
with $\Upphi_i=2\max\{\overline{\Upphi}^{\bu}_i,\overline{\Upphi}^{\bw}_i\}$
for $i=1,\ldots,6$.
\end{proof}

\subsection{Proof of Theorem~\ref{teo:global_estimates}}
\label{subsec:proof_of_teo:global_estimates}

The existence of $T_1$ and $\upkappa_1$ follows from 
\eqref{lem:u_ut_w_wt_main_estimate}.  Now, before starting the
proof of \eqref{teo:global_estimates_kapa2}-\eqref{teo:global_estimates_kapa14},
we deduce two estimates. First, differentiating \eqref{eq:ecuacion_continuidad}
with respect to $x_i$, using \eqref{eq:incompresibilidad}, testing the result 
by $|\rho_{x_i}|^{q-2}\rho_{x_i}$ and applying the Sobolev inequality we 
deduce that there exists $C_{sob}$ independent of $f$ and $g$ such that
\begin{eqnarray}
\frac{d}{dt}\|\nabla\rho(\cdot,t)\|^q_{\bL^q}\le C_{sob} 
\|\bu(\cdot,t)\|_{\bW^{2,s}}\|\nabla\rho(\cdot,t)\|_{\bL^q},\quad
\mbox{for $t\in [0,T_*]$.}
\label{eq:sobolev}
\end{eqnarray}
Second, by the regularity of the solutions for \eqref{eq:stokes_problem}
we have that there exists $C^{reg}_3$ depending only on $\mu,\mu_r$ and
$\Omega$ such that
\begin{eqnarray*}
\|\bu(\cdot,t)\|_{\bW^{2,s}}+\|p(\cdot,t)\|_{\bW^{1,s}}
\le C^{reg}_3\; \Big\|\Big(2\mu_r \curl \bw
+\rho f(\nabla h-\bm)
-\rho \bu_t-\rho (\bu\cdot\nabla) \bu\Big)(\cdot,t)\Big\|_{\bL^s}.
\end{eqnarray*}
Hence, by the Minkowski and H\"older inequalities and \eqref{eq:spoingag},
we find that there exists 
$\upxi_1=2\mu_rC^{reg}_3C_{gn},$
$\upxi_2=\beta C^{reg}_3\max\{C_{gn},\|\bm(\cdot,t)\|_{\bL^2}\}$, 
$\upxi_3=C^{reg}_3C_{gn}C^{2,\infty}_{iny}$ and 
$\upxi_4=\beta C^{reg}_3C_{poi}$, such that
\begin{eqnarray}
\|\bu(\cdot,t)\|_{\bW^{2,s}}+\|p(\cdot,t)\|_{W^{1,s}}
\le 
\upxi_1\|\bw(\cdot,t)\|_{\bH^2}
+\upxi_2|f(t)|\Big(\|h(\cdot,t)\|_{H^2}+1\Big)
\nonumber
\\
\hspace{2.8cm}
+\upxi_3\|\bu(\cdot,t)\|^2_{\bH^2}
+\upxi_4 \|\nabla \bu_t(\cdot,t)\|_{\bL^2},
\label{eq:w2s_regularity}
\end{eqnarray}
for $t\in [0,T_*]$.
Therefore, we derive the proof of  \eqref{teo:global_estimates_kapa2} by
inserting \eqref{eq:w2s_regularity} in \eqref{eq:sobolev} and using the estimates
\eqref{h_and_r_estimate}, \eqref{ht_and_rt_estimate} and \eqref{eq:lem:estimates3}. 
The estimate \eqref{teo:global_estimates_kapa3_4} is deduced
from \eqref{teo:global_estimates_kapa2} and
\eqref{h_and_r_estimate}.
The inequality \eqref{teo:global_estimates_kapa5_8} is obtained 
from \eqref{teo:global_estimates_kapa2}, \eqref{eq:uppi_cuatro}, \eqref{eq:uppi_cinco},
\eqref{ht_and_rt_estimate} and \eqref{eq:regularity_u_p_w}.
The estimate \eqref{teo:global_estimates_kapa9_10} is proved by the application of 
\eqref{teo:global_estimates_kapa1}, \eqref{eq:uppi_cuatro}, \eqref{eq:uppi_cinco} and
\eqref{eq:regularity_u_p_w}.
The estimate \eqref{teo:global_estimates_kapa11_13} follows
from \eqref{eq:ecuacion_continuidad}, \eqref{teo:global_estimates_kapa1}
and \eqref{teo:global_estimates_kapa9_10}.
We complete the proof of the theorem deducing the inequality
\eqref{teo:global_estimates_kapa14} 
by combining the results given on
\eqref{teo:global_estimates_kapa2} and \eqref{eq:w2s_regularity}.

\section{Well-posedness of the direct problem.}
\label{sec:dp}
 
The well-posedness of the direct problem is given
by the following Theorem:

\begin{theorem}
\label{teo:direct_problem_solution}
Consider that the functions $\rho_0,\bv_0,\bw_0,\bm,\bq,h$ and $r$ 
satisfy the hypothesis
(H$_1$)-(H$_4$) and $(f,g)\in [H^1(0,T)]^2$. Then, the direct problem
\eqref{eq:momento_lineal}-\eqref{eq:direct_problem_ic}
and \eqref{eq:helmholtz_f_uno}-\eqref{eq:helmholtz_f_cuatro}
possesses a unique solution $\{\bu,\bw,\rho,p,h\}$ in the sense of
the definition \ref{def:strong_solutions_dp}.
\end{theorem}

We note the proof of existence of solutions can be devolping by
applying the ideas of Boldrini~et. al. \cite{boldrini_2003}.
Meanwhile, the uniqueness is a straightforward consequence  
of a continuous dependence of the system unknowns with respect
to the source coefficients, which is proved in the following
Lemma.

\begin{lemma}
\label{lema:direct_problem_solution}
Consider that the functions $\rho_0,\bv_0,\bw_0$ and $h$ satisfying the hypothesis
(H$_1$)-(H$_4$). Suppose that $\{\bu_i,\bw_i,\rho_i,p_i,h_i\}$, $i=1,2$,
are two solutions of the direct problem
\eqref{eq:momento_lineal}-\eqref{eq:direct_problem_ic}
and \eqref{eq:helmholtz_f_uno}-\eqref{eq:helmholtz_f_cuatro}
corresponding to the coefficients $(f_i,g_i)\in [H^1(0,T)]^2$, $i=1,2$,
respectively.
There exists $C$ independent of  $(f_i,g_i)$, $i=1,2$, such that
the following estimate hold:
\begin{eqnarray}
&&\| \bu_1-\bu_2\|_{L^\infty([0,t];\bL^2(\Omega))}
+\|\bw_1-\bw_2\|_{L^\infty([0,t];\bL^2(\Omega))}
\nonumber\\&&
\hspace{3.1cm}
+\|\rho_1-\rho_2\|_{L^\infty([0,t];L^2(\Omega))}
+\|\nabla (h_1-h_2)\|_{L^\infty([0,t];\bL^2(\Omega))}
\nonumber\\&&
\hspace{3.1cm}
\le C\Big( \|f_1-f_2\|_{L^2([0,t])} + \|g_1-g_2\|_{L^2([0,t])}\Big)
\label{eq:continous_dependence}
\end{eqnarray}
for all $t$ belongs to the maximal interval where the solutions
are defined.

\end{lemma}

\begin{proof}
In order to simplify the presentation of the estimates
we introduce the following notation
\begin{eqnarray*}
&&\delta \bu=\bu_1-\bu_2,\quad 
\delta \rho=\rho_1-\rho_2,\quad
\delta h=h_1-h_2,\quad \delta \bw=\bw_1-\bw_2,
\\&&
\delta f=f_1-f_2,\quad
\delta g=g_1-g_2,\quad
\delta p=p_1-p_2
\quad\mbox{and}\quad
\overline{s}=2s/(s-2),
\end{eqnarray*}
where $s\in [2,\infty[$.
Hence,
by algebraic rearrangements  of both forward problems ($i=1$ and $i=2$), we get the
following system
\begin{eqnarray}
&&\rho_1\delta \bu_t
+\rho_1(\bu_1\cdot\nabla )\delta \bu
-(\mu+\mu_r)\Delta\delta \bu
+\nabla \delta p
=-\delta  \rho (\bu_2)_t
+\Big[(\delta \rho \bu_1+\rho_2 \delta \bu)\cdot\nabla \Big]\bu_2
\nonumber
\\&&
\hspace{2cm}
+2\mu_r\curl \delta \bw 
+\delta \rho f_1\nabla h_1
+\rho_2 \delta f\nabla h_1
\nonumber
\\&&
\hspace{2cm}
+\rho_2 f_2 \nabla \delta h
-\Big(\delta \rho f_1+\rho_2 \delta f\Big)\bm,
\quad\mbox{in}\quad Q_T,
\label{eq:mom_lin_dif}
\\&& 
\diver\;\delta \bu=0,
	\quad\mbox{in}\quad Q_T,
\label{eq:incom_dif}
\\&&
\rho_1\delta \bw_t
+\rho_1(\bw_1\cdot\nabla )\delta \bw
-(c_a+c_d)\Delta\delta \bw
=-\delta \rho (\bw_2)_t
+\Big[(\delta \rho \bw_1+\rho_2\delta \bw)\cdot\nabla \Big]\bw_2-4\mu_r\delta \bw
\nonumber
\\&&
\hspace{2cm}
	+(c_o+c_d-c_a)\nabla\diver\delta \bw 
	+2\mu_r \curl\delta \bu
	+\Big(\delta \rho g_1+\rho_2 \delta g\Big)\bq,
	\quad\mbox{in}\quad Q_T,
\label{eq:mom_ang_dif}
\\&&
 \delta \rho_t+\delta \bu\cdot \nabla \rho_1+\bu_2\cdot \nabla \delta \rho=0,
	\quad\mbox{in}\quad Q_T,
\label{eq:ec_con_dif}
\\&&
\delta \bu(\bx,t)=\delta \bw(\bx,t)=0,
	\quad\mbox{on}\quad \Sigma_T,
\label{eq:dir_prob_bc_dif}
\\&&
\delta \bu(\bx,0)=\delta \bw(\bx,0)=\delta \rho(\bx,0)=0,
	\quad\mbox{on}\quad \Omega,
\label{eq:dir_prob_ic_dif}
\\&&
\diver(\rho_1\nabla \delta h)=\diver(\delta \rho (\bm-\nabla h_2)),
\;\mbox{in}\quad\Omega,
	\label{eq:helm_f_dos_dif}
\\&&
\frac{\partial \delta h}{\partial \bn}(\bx,t)=0,
	\;\mbox{on}\quad \Sigma_T,
	\label{eq:helm_f_tres_dif}
\\&&
\int_{\Omega}\delta h(\bx,t)d\bx=0, 
	\;t\in[0,T].
	\label{eq:helm_f_cuatro_dif}
\end{eqnarray}
Now,  to prove \eqref{eq:continous_dependence} we proceed in two big steps: first we obtain
five a priori estimates for 
the system \eqref{eq:mom_lin_dif}-\eqref{eq:helm_f_cuatro_dif}
and then we apply the Gronwall inequality.

First, by equations \eqref{eq:incompresibilidad} for $u_2$ 
and \eqref{eq:ec_con_dif}; 
the boundary condition \eqref{eq:dir_prob_bc_dif};
integration by parts; the H\"older and Young's inequalities; 
and \eqref{eq:spoingag}
for $p=2$ and $q=\overline{s}\in ]2,6[$, we have that
\begin{eqnarray}
 \frac{1}{2}\frac{d}{dt}\int_{\Omega}\delta \rho^2(\bx,t)d\bx
&=&\int_{\Omega}(\delta \rho_t\delta \rho)(\bx,t)d\bx
=\int_{\Omega}-\Big[\delta \bu\cdot\nabla\rho_1
	+\bu_2\cdot\nabla\delta \rho\Big]\delta \rho(\bx,t)d\bx
\nonumber
\\
&=&
-\int_{\Omega}\delta \rho(\delta \bu\cdot\nabla\rho_1)(\bx,t)d\bx
+\frac{1}{2}\int_{\Omega}\delta \rho^2\diver(\bu_2)(\bx,t)d\bx
\nonumber
\\
&\le&
\|\nabla\rho_1(\cdot,t)\|_{L^s(\Omega)}
\|\delta \bu(\cdot,t)\|_{L^{\overline{s}}(\Omega)}
\|\delta \rho(\cdot,t)\|_{L^2(\Omega)}.
\nonumber
\\
&\le&
 \frac{1}{2\epsilon_\bu}\|\nabla\rho_1(\cdot,t)\|^2_{L^s(\Omega)}
\|\delta \rho(\cdot,t)\|^2_{L^2(\Omega)}
+
\frac{\epsilon_\bu}{2}C_{poi}
\|\nabla\delta \bu(\cdot,t)\|^2_{L^2(\Omega)},
\label{lema:direct_problem_solution:ec1}
\end{eqnarray}
where $\epsilon_\bu>0$ is the parameter used for the Young's inequality.
In the second place, by equation \eqref{eq:ecuacion_continuidad}
for $\rho_1$ and the boundary condition \eqref{eq:dir_prob_bc_dif},
we note that the left hand side of 
\eqref{eq:mom_lin_dif} multiplied by $\delta \bu$ 
can be integrated by parts and  simplified as follows
\begin{eqnarray}
&&
\int_{\Omega}\Big[\rho_1\delta \bu_t
+\rho_1(\bu_1\cdot\nabla )\delta \bu
-(\mu+\mu_r)\Delta\delta \bu
+\nabla \delta p\Big]\cdot\delta \bu (x,t) dx
\nonumber\\
&&
\hspace{0.5cm}
=\int_{\Omega}\Big[\frac{1}{2}\Big\{(\rho_1|\delta \bu|^2)_t-(\rho_1)_t|\delta \bu|^2\Big\}
+\rho_1[(\bu_1\cdot\nabla )\delta \bu]\cdot\delta \bu
\nonumber\\
&&
\hspace{1.5cm}
-(\mu+\mu_r)\Delta\delta \bu\cdot\delta \bu
+\nabla \delta p\cdot\delta \bu\Big](\bx,t)d\bx
\nonumber
\\
&&
\hspace{0.5cm}
=\frac{1}{2}\frac{d}{dt}
\int_{\Omega}(\rho_1|\delta \bu|^2)(\bx,t)d\bx
+(\mu+\mu_r)\int_{\Omega}|\nabla\delta \bu|^2(\bx,t)d\bx.
\label{lema:direct_problem_solution:ec3}
\end{eqnarray}
Then, a multiplication of \eqref{eq:mom_lin_dif} by $\delta u$,
integration by parts and  application of the H\"older inequality leads to 
\begin{eqnarray}
&&
\frac{1}{2}\frac{d}{dt}
\int_{\Omega}(\rho_1|\delta \bu|^2)(\bx,t)d\bx
+(\mu+\mu_r)\int_{\Omega}|\nabla\delta \bu|^2(\bx,t)d\bx
\nonumber
\\
&&
\hspace{0.3cm}
\le
\|(\bu_2)_t(\cdot,t)\|_{\bL^s}\|\delta \rho(\cdot,t)\|_{\bL^2}
	\|\delta \bu(\cdot,t)\|_{\bL^s}
+\|\bu_1(\cdot,t)\|_{\bL^{\infty}}\|\nabla \bu_2(\cdot,t)\|_{\bL^s}
	\|\delta \rho(\cdot,t)\|_{\bL^2}
	\|\delta \bu(\cdot,t)\|_{\bL^{\overline{s}}}
\nonumber
\\
&&
\hspace{0.3cm}+
\|\rho_2(\cdot,t)\|_{\bL^\infty}\|\nabla \bu_2(\cdot,t)\|_{\bL^s}
	\|\delta \bu(\cdot,t)\|_{\bL^2}
	\|\delta \bu(\cdot,t)\|_{\bL^{\overline{s}}}
+2\mu_r \|\nabla\delta \bw(\cdot,t)\|_{\bL^2} \|\delta \bu(\cdot,t)\|_{\bL^2}
\nonumber
\\
&&
\hspace{0.3cm}+
|f_1(t)|\|\nabla h_1(\cdot,t)\|_{\bL^s}\|\delta \rho(\cdot,t)\|_{\bL^2}
	\|\delta \bu(\cdot,t)\|_{\bL^{\overline{s}}}
\nonumber
\\
&&
\hspace{0.3cm}	
+|\delta f(t)|\|\nabla h_1(\cdot,t)\|_{\bL^s}\|\rho_2(\cdot,t)\|_{\bL^\infty}
	\|\delta \bu(\cdot,t)\|_{\bL^2}|\Omega|^{1/\overline{s}}
\nonumber
\\
&&
\hspace{0.3cm}+
|f_2(t)|\|\nabla \delta h(\cdot,t)\|_{\bL^2}\|\rho_2(\cdot,t)\|_{\bL^\infty}
	\|\delta \bu(\cdot,t)\|_{\bL^2}
	+|f_1(t)|\|\bm(\cdot,t)\|_{\bL^\infty}
	\|\delta \rho(\cdot,t)\|_{\bL^2}\|\delta \bu(\cdot,t)\|_{\bL^2}
\nonumber
\\
&&
\hspace{0.3cm}+
|\delta f(t)|\|\bm(\cdot,t)\|_{\bL^\infty}
\|\rho_2(\cdot,t)\|_{L^\infty}\|\delta \bu(\cdot,t)\|_{\bL^2}
|\Omega|^{1/2}.
\label{lema:direct_problem_solution:ec4}
\end{eqnarray}
Hence by the Young's inequality and \eqref{eq:spoingag}, 
we have that 
there exists $\Gamma^\bu_i:=\Gamma^\bu_i(t)>0$, $i\in\{1,2,3\}$, such that
(see \ref{app:constantes_gamma} for details)
\begin{eqnarray}
&&
\frac{1}{2}\frac{d}{dt}
\int_{\Omega}(\rho_1|\delta \bu|^2)(\bx,t)d\bx
+(\mu+\mu_r)\int_{\Omega}|\nabla\delta \bu|^2(\bx,t)d\bx
\nonumber
\\
&&
\hspace{2cm}
\le
\Gamma^u_1 \Big(\|\sqrt{\rho_1}\delta \bu(\cdot,t)\|^2_{\bL^2(\Omega)}
	+\|\delta \rho(\cdot,t)\|^2_{\bL^2(\Omega)}\Big)
	+\Gamma^\bu_2\|\nabla \delta h(\cdot,t)\|^2_{\bL^2(\Omega)}
	+\Gamma^\bu_3|\delta f(t)|^2
\nonumber
\\
&&
\hspace{2.3cm}
+2\epsilon_\bu C_{poi}\;\|\nabla\delta \bu(\cdot,t)\|^2_{\bL^2(\Omega)}
+\frac{\epsilon_\bw}{2}\;\|\nabla\delta \bw(\cdot,t)\|^2_{\bL^2(\Omega)}.
\label{lema:direct_problem_solution:ec5}
\end{eqnarray}
For the third estimate, we start from \eqref{eq:mom_ang_dif}
and proceeding  by a similar reasoning to the steps 
\eqref{lema:direct_problem_solution:ec3}-\eqref{lema:direct_problem_solution:ec5}
we deduce that
there exists $\Gamma^\bw_i:=\Gamma^\bw_i(t)>0$, $i\in\{1,2\}$, such that
(see \ref{app:constantes_gamma})
\begin{eqnarray}
&&
\frac{1}{2}\frac{d}{dt}
\int_{\Omega}(\rho_1|\delta \bw|^2)(\bx,t)d\bx
+(c_a+c_d)\int_{\Omega}|\nabla\delta \bw|^2(\bx,t)d\bx
\\
&&
\hspace{2cm}
+(c_0+c_d-c_a)\int_{\Omega}|\diver\delta \bw|^2(\bx,t)d\bx
+4u_r\int_{\Omega}|\delta \bw|^2(\bx,t)d\bx
\nonumber
\\
&&
\hspace{2cm}
\le
\Gamma^\bw_1 \Big(\|\sqrt{\rho_1}\delta \bw(\cdot,t)\|^2_{\bL^2(\Omega)}
+\|\delta \rho(\cdot,t)\|^2_{\bL^2(\Omega)}\Big)
+\Gamma^\bw_2|\delta g(t)|^2
\nonumber
\\
&&
\hspace{2.3cm}
+\frac{\epsilon_\bu}{2}\;\|\nabla\delta \bu(\cdot,t)\|^2_{\bL^2(\Omega)}
+2\epsilon_\bw C_{poi}\;\|\nabla\delta \bw(\cdot,t)\|^2_{\bL^2(\Omega)}.
\label{lema:direct_problem_solution:ec6}
\end{eqnarray}
In the fourth place, we deduce an estimate related to $\delta h$.
Indeed, by equations for $\delta h$ given on 
\eqref{eq:helm_f_dos_dif}-\eqref{eq:helm_f_cuatro_dif},
integration by parts,
the H\"older and Minkowski inequalities,
and Lemma~\ref{lema:aprioriestimatesforrho},
we deduce that
\begin{eqnarray*}
&&\alpha\|\nabla \delta h(\cdot,t)\|^2_{\bL^2(\Omega)}
\le \int_{\Omega} \Big(\rho_1 |\nabla \delta h|^2\Big)(\bx,t) d\bx
=-\int_{\Omega}\diver(\rho_1 \nabla \delta h)\delta h (\bx,t)d\bx
\nonumber\\&&
\hspace{2cm}
=-\int_{\Omega}\Big[\diver(\delta \rho (\bm-\nabla h_2))\delta h\Big](\bx,t) d\bx
=\int_{\Omega}\Big[\delta \rho (\bm-\nabla h_2)\nabla \delta h\Big](\bx,t) d\bx
\nonumber\\&&
\hspace{2cm}
\le \|\delta \rho (\bm-\nabla h_2)(\cdot,t)\|_{\bL^2(\Omega)}
\|\nabla \delta h (\cdot,t)\|_{\bL^2(\Omega)}
\nonumber\\&&
\hspace{2cm}
\le 
\Big(\|\bm(\cdot,t)\|_{\bL^\infty(\Omega)}
+\|\nabla h_2(\cdot,t)\|_{\bL^\infty(\Omega)}\Big)
\|\delta \rho(\cdot,t)\|_{L^2(\Omega)}
\|\nabla \delta h(\cdot,t) \|_{\bL^2(\Omega)}.
\end{eqnarray*}
Thus, we have the following two estimates
\begin{eqnarray}
\|\nabla \delta h(\cdot,t)\|_{\bL^2(\Omega)}
\le 
\frac{1}{\alpha}
\Big(\|\bm(\cdot,t)\|_{\bL^\infty(\Omega)}
+\|\nabla h_2(\cdot,t)\|_{\bL^\infty(\Omega)}\Big)
\|\delta \rho(\cdot,t)\|_{L^2(\Omega)}.
\label{lema:direct_problem_solution:ec7}
\end{eqnarray}

\vspace{0.5cm}
From 
\eqref{lema:direct_problem_solution:ec1} 
and \eqref{lema:direct_problem_solution:ec5},
selecting $\epsilon_\bu=2(\mu+\mu_r)(5C_{poi}+1)^{-1}$
and $\epsilon_\bw=2(c_a+c_d)(5C_{poi}+1)^{-1}$,
we deduce that
\begin{eqnarray*}
&&\frac{d}{dt}
\int_{\Omega}(\rho_1|\delta \bu|^2+\rho_1|\delta \bw|^2+\delta \rho^2)(\bx,t)d\bx
+\int_{\Omega}(\nabla\delta h)^2(\bx,t)d\bx
\\&&
\hspace{2cm}
\le 
\Gamma^\delta(t)
\int_{\Omega}(\rho_1|\delta \bu|^2+\rho_1|\delta \bw|^2+\delta \rho^2)(\bx,t)d\bx
+\Gamma^\bu_3(t)|\delta f(t)|^2
+\Gamma^\bw_3(t)|\delta g(t)|^2.
\end{eqnarray*}
where 
$
\Gamma^\delta(t)=\max\{\Gamma^\bu_1(t),\Gamma^\bw_1(t),C^\delta(t)\},
$
with $C^\delta$ defined as follows
\begin{eqnarray*}
C^\delta(t)&=&\frac{1}{2\epsilon_\bu}\|\nabla \rho_1(\cdot,t)\|_{\bL^2(\Omega)}
+\Gamma^\bu_1(t)+\frac{\Gamma^\bu_2(t)}{\alpha^2}
\Big(\|\bm(\cdot,t)\|_{\bL^\infty(\Omega)}
+\|\nabla h_2(\cdot,t)\|_{\bL^\infty(\Omega)}\Big)^2
\\
&&+\Gamma^\bw_1(t)+\frac{\Gamma^\bw_2(t)}{\alpha^2}
\Big(\|\bq(\cdot,t)\|_{\bL^\infty(\Omega)}
+\|\nabla g_2(\cdot,t)\|_{\bL^\infty(\Omega)}\Big)^2.
\end{eqnarray*}
We note that 
$\|\Gamma^\delta\|_{L^1(0,T)}<\infty $ and $\Gamma^\bu_3(t)<\infty$.
Hence, we complete the proof of \eqref{eq:continous_dependence} by application of the
Gronwall inequality.
\end{proof}

\section{Well-posedness of the inverse problem}
\label{sec:ip}

The main result of this section is the derivation of
a well-posedness result for the inverse problem, which is 
formalized by the following theorem:
\begin{theorem}
\label{teo:wellposeIP}
Assume that (H$_1$)-(H$_7$) are satisfied. Then there exists 
a unique solution $\{\bu,\bw,\rho,p,h,f,g\}$ of the inverse problem
defined on a small enough time $T_{\star}$.
\end{theorem}

The proof of Theorem~\ref{teo:wellposeIP} is detailed on 
subsection~\ref{subsec:proofteoIP}. It is based on a reformulation
of the inverse problem like as an operator equation of second kind 
and then by application of the fixed point argument.

\subsection{Formulation of the inverse problem
\eqref{eq:momento_lineal}-\eqref{eq:helmholtz_f_cuatro}
as and operator equation}
\label{section:inv_operator_form}

We define the nonlinear operator
\begin{eqnarray}
\begin{array}{rcclc}
\mathcal{R}&:&[H^1(0,T)]^2&\to&[H^1(0,T)]^2
\\
& &(f,g) &\mapsto & \displaystyle
(\mathcal{R}_1,\mathcal{R}_2):
	    =\left(\frac{\mathcal{N}_1}{\gamma_1},
	    \frac{\mathcal{N}_2}{\gamma_2}\right)
\end{array}
\label{eq:operador}
\end{eqnarray}
where 
\begin{eqnarray}
 \mathcal{N}_1(t)&=&
	\frac{d\phi^{\bu}}{dt}(t)-\int_{\Omega}\rho \bu\otimes \bu:\nabla \bpsi^{\bu} dx
	+(\mu +\mu_r)\Big(A \bu,\bpsi^{\bu} \Big)
	-2\mu_r \Big(\curl \bw,\bpsi^{\bu}\Big),
\label{eq:operador_N1}
\\
\mathcal{N}_2(t)&=&
	\frac{d\phi^{\bw}}{dt}(t)-\int_{\Omega}\rho \bw\otimes \bw:\nabla \bpsi^{\bw} dx
	+\Big(L \bw,\bpsi^{\bw} \Big)
	-2\mu_r \Big(\curl \bu,\bpsi^{\bw}\Big),
\label{eq:operador_N2}
\\
 \gamma_1(t)&=&\int_{\Omega}\rho(x,t)(\nabla h(x,t)- \bm(x,t))\cdot\bpsi^{\bu}(x)dx
	\quad\mbox{and}
\label{eq:operador_gamma1}
\\
 \gamma_2(t)&=&\int_{\Omega}\rho(x,t)\bq(x,t)\cdot\bpsi^{\bw}(x)dx.
\label{eq:operador_gamma2}
\end{eqnarray}
We note that the solvability of the inverse problem 
\eqref{eq:momento_lineal}-\eqref{eq:helmholtz_f_cuatro} is connected with the following
operator equation of the second kind 
\begin{eqnarray}
\hspace{1cm}
\mathcal{R}(f,g)=(f,g)
\quad
\mbox{over}
\quad
D:=\{(f,g)\in [H^1(0,T)]^2:\|(f,g)\|_{[H^1(0,T)]^2}\le \upeta\;\},
\label{eq:operador_ecuation}
\end{eqnarray}
where $\upeta\in\mathbb{R}^+$ will be constructed  in order to get
the solvability of~\eqref{eq:operador_ecuation} by application
of the fixed point argument.
More specifically, we have the following  characterization result
of the solvability of the inverse problem. 

\begin{theorem}
\label{teo:ip_vs_operator_ec}
Let (H$_1$)-(H$_7$) be satisfied. Then the following two
assertions are valid:
\begin{enumerate}
\item[(i)] If the inverse problem 
\eqref{eq:momento_lineal}-\eqref{eq:helmholtz_f_cuatro}
is solvable, then so is equation \eqref{eq:operador_ecuation}, and
\item[(ii)] If \eqref{eq:operador_ecuation}  is solvable 
and the following compatibility conditions
\begin{eqnarray}
\int_{\Omega}\rho(\bx,0)\bu(\bx,0)\cdot\bpsi^{\bu}(\bx)d\bx=\phi^{\bu}(0)
	\quad\mbox{and}\quad
	\int_{\Omega}\rho(\bx,0)\bw(\bx,0)\cdot\bpsi^{\bw}(\bx)d\bx=\phi^{\bw}(0)
\label{eq:compat_cond}
\end{eqnarray}
are satisfied, then there exists a solution of 
the inverse problem \eqref{eq:momento_lineal}-\eqref{eq:helmholtz_f_cuatro}.
\end{enumerate}
\end{theorem}

\begin{proof}
{\it (i)} Let us consider $\{\bu,\bw,\rho,p,h,f,g\}$ a solution of the inverse problem
\eqref{eq:momento_lineal}-\eqref{eq:helmholtz_f_cuatro}. By the
overdetermination condition \eqref{eq:direct_problem_ovc} and the definitions
of $\gamma_1$ and $\gamma_2$, given on \eqref{eq:operador_gamma1}
and \eqref{eq:operador_gamma2}, we deduce
the following identities 
\begin{eqnarray*}
&&\Big((\rho \bu)_t,\psi^\bv\Big)=\frac{d\phi^{\bu}}{dt}(t),
\quad 
\Big((\rho \bw)_t,\bpsi^{\bw}\Big)=\frac{d\phi^{\bw}}{dt}(t),
\\&&
\Big(\rho \bF,\psi^\bv\Big)=f(t)\gamma_1(t)
\quad
\mbox{and}
\quad
\Big(\rho \bG,\bpsi^{\bw}\Big)=g(t)\gamma_2(t).
\end{eqnarray*}
Hence, the selection $\bv=\bpsi^{\bu}$ and $\bvarphi=\bpsi^{\bw}$ on
the integral
identities \eqref{eq:formulacion_variacional_u}-\eqref{eq:formulacion_variacional_w}
leads to
\begin{eqnarray}
&&\frac{d\phi^{\bu}}{dt}(t)-\int_{\Omega}\rho \bu\otimes \bu:\nabla \bpsi^{\bu} d\bx
	+(\mu +\mu_r)\Big(A \bu,\bpsi^{\bu} \Big)
	-2\mu_r \Big(\curl \bw,\bpsi^{\bu}\Big)=f(t)\gamma_1(t),
\label{eq:operator_eq_i}
\\
&&\frac{d\phi^{\bw}}{dt}(t)-\int_{\Omega}\rho \bw\otimes \bw:\nabla \bpsi^{\bw} d\bx
	+\Big(L \bw,\bpsi^{\bw} \Big)
	-2\mu_r \Big(\curl \bu,\bpsi^{\bw}\Big)=g(t)\gamma_2(t).
\label{eq:operator_eq_ii}
\end{eqnarray}
Clearly, in view of \eqref{eq:operator_eq_i}-\eqref{eq:operator_eq_ii}, we
conclude that $(f,g)$ solves the operator equation~\eqref{eq:operador_ecuation}.

\noindent
{\it (ii)}. Let us consider 
$(f,g)\in [H^1(0,T)]^2$ a solution of 
the operator equation~\eqref{eq:operador_ecuation}.
Then the relations \eqref{eq:operator_eq_i} and \eqref{eq:operator_eq_ii}
are satisfied and there exists 
$\{\bu,\bw,\rho,p,h\}$ a 
solution of the direct problem \eqref{eq:momento_lineal}-\eqref{eq:direct_problem_ic}
and \eqref{eq:helmholtz_f_uno}-\eqref{eq:helmholtz_f_cuatro}.
By selecting $\bv=\bpsi^{\bu}$ and $\bvarphi=\bpsi^{\bw}$ on the integral
identities \eqref{eq:formulacion_variacional_u}-\eqref{eq:formulacion_variacional_w},
we arrive at
\begin{eqnarray}
&&\Big((\rho \bu)_t,\bpsi^{\bu}\Big)
	-\int_{\Omega}\rho \bu\otimes \bu:\nabla \bpsi^{\bu} d\bx
	+(\mu +\mu_r)\Big(A \bu,\bpsi^{\bu} \Big)
-2\mu_r \Big(\curl \bw,\bpsi^{\bu}\Big)=
	f(t)\gamma_1(t),\qquad\quad
\label{eq:operator_eq_iii}
\\&&
\Big((\rho \bw)_t,\bpsi^{\bw}\Big)
	-\int_{\Omega}\rho \bw\otimes \bw:\nabla \bpsi^{\bw} dx
	+\Big(L \bw,\bpsi^{\bw} \Big)
-2\mu_r \Big(\curl \bu,\bpsi^{\bw}\Big)
=g(t)\gamma_2(t).
\label{eq:operator_eq_iv}
  \end{eqnarray}
Subtracting, \eqref{eq:operator_eq_iii} from \eqref{eq:operator_eq_i}
and \eqref{eq:operator_eq_iv} from \eqref{eq:operator_eq_ii} we deduce
that
\begin{eqnarray}
\Big((\rho \bu)_t,\bpsi^{\bu}\Big)-\frac{d\phi^{\bu}}{dt}(t)=0
	\quad\mbox{and}\quad
\Big((\rho \bw)_t,\bpsi^{\bw}\Big)-\frac{d\phi^{\bw}}{dt}(t)=0.
\label{eq:operator_eq_v}
  \end{eqnarray} 
Integrating the equations \eqref{eq:operator_eq_v} and using the 
compatibility conditions given on \eqref{eq:compat_cond}, we deduce that the
overdetermination conditions \eqref{eq:direct_problem_ovc}
are satisfied. Thus, we have that  $\{\bu,\bw,\rho,p,h,f,g\}$ is a
solution of  the inverse problem 
\eqref{eq:momento_lineal}-\eqref{eq:helmholtz_f_cuatro}.  
\end{proof}

\subsection{Properties of the operator $\mathcal{R}$ 
defined on \eqref{eq:operador}-\eqref{eq:operador_gamma2}}

\begin{lemma}
\label{lem:Rmapitself}
Assume that (H$_1$)-(H$_7$) are satisfied. Then there exists 
a small enough time $T_{\star}$ and there exists $\upeta\in\R^+$
such that $\mathcal{R}$ maps $D$ into itself.
\end{lemma} 

\begin{proof}
From (H$_7$) and Thoerem~\ref{teo:direct_problem_solution} we deduce
that there exists $T_{\bowtie}\in [0,T_*]$ such that
\begin{eqnarray}
\gamma_1,\gamma_2\in C[0,T_{\bowtie}], 
\quad
|\gamma_1(t)|\ge h^\epsilon>0
\quad
\mbox{and}
\quad
|\gamma_2(t)|\ge r^\epsilon>0
\quad
\mbox{on}
\quad
[0,T_{\bowtie}].
\label{eq:remark_h7}
\end{eqnarray}
Then, by the definition of $\mathcal{R}$ given on \eqref{eq:operador}, we have that
\begin{eqnarray}&&
 \|\mathcal{R}(f,g)\|^2_{[H^1([0,T])]^2}
=
\sum_{i=1}^2\|\mathcal{R}_i(f,g)\|^2_{H^1([0,T])}
\nonumber\\&&
\qquad
\le\frac{1}{(h^\epsilon)^2}\|\mathcal{N}_1\|^2_{L^2([0,T])}
+\left(
\frac{1}{(h^\epsilon)^2}\left\|\frac{d\mathcal{N}_1}{dt}\right\|^2_{L^2([0,T])}
+\frac{1}{(h^\epsilon)^4}\|\mathcal{N}_1\|_{L^2([0,T])}
\left\|\frac{d\gamma_1}{dt}\right\|_{L^2([0,T])}
\right)^2
\nonumber\\&&
\qquad
+\frac{1}{(r^\epsilon)^2}\|\mathcal{N}_2\|^2_{L^2([0,T])}
+\left(
\frac{1}{(r^\epsilon)^2}\left\|\frac{d\mathcal{N}_2}{dt}\right\|^2_{L^2([0,T])}
+\frac{1}{(r^\epsilon)^4}\|\mathcal{N}_2\|_{L^2([0,T])}
\left\|\frac{d\gamma_2}{dt}\right\|_{L^2([0,T])}
\right)^2.\qquad
\label{eq:remark_h7:uno}
\end{eqnarray}
Now, denoting by $\|\cdot\|_{p,q}$ 
the norm $\|\cdot\|_{L^p([0,T];L^q(\Omega))}$ 
and by $\|\cdot\|_{\bL^p}$ 
the norm $\|\cdot\|_{\bL^p((\Omega)}$, making use of
the relation \eqref{eq:ident_stokes} for the operator $A$ 
and applying the Minkowski and
H\"older inequalities, we have that
\begin{eqnarray*}
&&\|\mathcal{N}_1\|_{L^2([0,T])}
\le\left\|\frac{d\phi^{\bu}}{dt}\right\|_{L^2([0,T])}+
	\|\rho\|_{\infty,\infty}\|\bu\|^2_{\infty,2}
	\|\bpsi^{\bu}\|_{\bL^\infty}T^{1/2}
	+(\mu +\mu_r)\|\nabla \bu\|_{\infty,2}
	\times
\\&&
\hspace{1.5cm}
	\|\nabla \bpsi^{\bu}\|_{\bL^2}T^{1/2}
	+2\mu_r \|\nabla \bw\|_{\infty,2}
	\|\bpsi^{\bu}\|_{\bL^2}T^{1/2},
\\&&
\|\mathcal{N}_2\|_{L^2([0,T])}
\le\left\|\frac{d\phi^{\bw}}{dt}\right\|_{L^2([0,T])}+
	\|\rho\|_{\infty,\infty}\|\bw\|^2_{\infty,2}
	\|\bpsi^{\bw}\|_{\bL^\infty}T^{1/2}
	+ \|\bw\|_{\infty,2}
	\|L\bpsi^{\bw}\|_{\infty,2} 
\\&&
\hspace{1.5cm}
	+2\mu_r \|\nabla \bw\|_{\infty,2}
	\|\bpsi^{\bw}\|_{\bL^2}T^{1/2},
\\&&
\left\|\frac{d\mathcal{N}_1}{dt}\right\|_{L^2([0,T])}
\le\left\|\frac{d^2\phi^{\bu}}{dt^2}\right\|_{L^2([0,T])}+
	\|\rho_t\|_{\infty,2}\|\bu\|^2_{\infty,6}
	\|\bpsi^{\bu}\|_{\bL^6}T^{1/2}
	+
	2\|\rho\|_{\infty,\infty}
	\|\bu_t\|_{\infty,2}\|\bu\|_{\infty,6}
	\times
\\&&
\hspace{1.5cm}
	\|\bpsi^{\bu}\|_{\bL^3}T^{1/2}
	  +(\mu +\mu_r)\|\nabla \bu_t\|_{\infty,2}
	\|\nabla\bpsi^{\bu}\|_{\bL^2}T^{1/2}
	+2\mu_r \|\nabla \bw_t\|_{\infty,2}
	\|\bpsi^{\bu}\|_{\bL^2}T^{1/2},
\\&&
\left\|\frac{d\mathcal{N}_2}{dt}\right\|_{L^2([0,T])}
\le\left\|\frac{d^2\phi^{\bw}}{dt^2}\right\|_{L^2([0,T])}+
	\|\rho_t\|_{\infty,2}\|w\|^2_{\infty,6}
	\|\bpsi^{\bw}\|_{\bL^6}T^{1/2}
	+2\|\rho\|_{\infty,\infty}
	\|w_t\|_{\infty,2}\|w\|_{\infty,6}
	\times
\\&&
\hspace{1.5cm}
	\|\bpsi^{\bw}\|_{\bL^3}T^{1/2}
	+\|\bw_t\|_{\infty,2}
	\|L\bpsi^{\bw}\|_{\bL^2} 
	+2\mu_r \|\nabla \bw_t\|_{\infty,2}
	\|\bpsi^{\bu}\|_{\bL^2}T^{1/2},
\\&&
\left\|\frac{d\gamma_1}{dt}\right\|_{L^2([0,T])}\le
	\|\rho_t\|_{\infty,2}
	\|\nabla h-\bm\|_{\infty,2}
	\|\bpsi^{\bu}\|_{\bL^\infty}T^{1/2}
	+\|\rho\|_{\infty,\infty}
	\|\nabla h_t-\bm_t\|_{\infty,2}
	\|\bpsi^{\bu}\|_{\bL^2}T^{1/2}\;\mbox{and}
\\&&
\left\|\frac{d\gamma_2}{dt}\right\|_{L^2([0,T])}\le
	\|\rho_t\|_{\infty,2}
	\|\bq\|_{\infty,2}
	\|\bpsi^{\bw}\|_{\bL^\infty}T^{1/2}
	+\|\rho\|_{\infty,\infty}
	\|\nabla \bq_t\|_{\infty,2}
	\|\bpsi^{\bw}\|_{\bL^2}T^{1/2}.
\end{eqnarray*}
Hence, by Theorem~\ref{teo:global_estimates} and 
Lemmas~\ref{lema:aprioriestimatesforrho}-\ref{lema:aprioriestimatesforhandr}, we have that
\begin{eqnarray*}
\|\mathcal{N}_i\|_{L^2([0,T])}\le \eta_i,\quad
\|d\mathcal{N}_i/dt\|_{L^2([0,T])}\le \overline{\eta}_i
\quad\mbox{and}\quad
\|d\gamma_i/dt\|_{L^2([0,T])}\le \sigma_i,
\quad i=1,2,
\end{eqnarray*}
where
\begin{eqnarray*}
\eta_1&=&
	\left\|\frac{d\phi^{\bu}}{dt}\right\|_{L^2([0,T])}
	+\upkappa_1T^{1/2}
	\Big(\beta \upkappa_1\|\bpsi^{\bu}\|_{\bL^\infty}
	+(\mu+\mu_r)\|\nabla\bpsi^{\bu}\|_{\bL^2}
	+2\mu_r\|\bpsi^{\bu}\|_{\bL^2}\Big),
 \\
\eta_2&=&\left\|\frac{d\phi^{\bw}}{dt}\right\|_{L^2([0,T])}
	+\upkappa_1 T^{1/2}
	\Big(\beta \upkappa_1\|\bpsi^{\bw}\|_{\bL^\infty}
	+ \|L\bpsi^{\bw}\|_{\bL^2}
	+2\mu_r\|\bpsi^{\bw}\|_{\bL^2}\Big),
\\
\overline{\eta}_1&=&
	\left\|\frac{d^2\phi^{\bu}}{dt^2}\right\|_{L^2([0,T])}
	+T^{1/2}
	\Big(\upkappa_{8} (C_{poi}\upkappa_1)^2\|\bpsi^{\bu}\|_{\bL^6}
	+2\beta \upkappa_1\upkappa_2C_{poi}\|\bpsi^{\bu}\|_{\bL^3}
\\
&&
\hspace{2.4cm}
	+(\mu+\mu_r)\upkappa_2\|\nabla\bpsi^{\bu}\|_{\bL^2}
	+2\mu_r\upkappa_2\|\bpsi^{\bu}\|_{L^2}\Big),
 \\
\overline{\eta}_2&=&
	\left\|\frac{d^2\phi^{\bw}}{dt^2}\right\|_{L^2([0,T])}
	+T^{1/2}
	\Big(\upkappa_{8} (C_{poi}\upkappa_1)^2\|\bpsi^{\bw}\|_{\bL^6}
	+2\beta \upkappa_1\upkappa_2C_{poi}\|\bpsi^{\bw}\|_{\bL^3}
\\&&
\hspace{2.4cm}
	+\upkappa_1\|L\bpsi^{\bw}\|_{\bL^2}
	+2\mu_r\upkappa_2\|\bpsi^{\bw}\|_{\bL^2}\Big),
\\
\sigma_1&=&\Big[\upkappa_{8}\Big(\Uppi_1+1\Big)\|\bm\|_{\infty,2}
	\|\bpsi^{\bu}\|_{L^\infty}+\beta\Big(
	\upkappa_{4}+\upkappa_{5}\upeta+\|\bm_t\|_{\infty,2}\Big)
	\|\bpsi^{\bu}\|_{\bL^2}\Big]\quad\mbox{and}
\\
\sigma_2&=&\Big[\upkappa_{8}\|\bq\|_{\infty,2}
	\|\bpsi^{\bw}\|_{\bL^\infty}+\beta\|\bq_t\|_{\infty,2}
	\|\bpsi^{\bw}\|_{\bL^2}\Big].
\end{eqnarray*}
Thus, by \eqref{eq:remark_h7:uno}, we follow that there exists 
$\Theta_i,\;i=1,2,3$ independent of $\upeta$ such that
$\|\mathcal{R}(f,g)\|^2_{[H^1([0,T])]^2}\le \Theta_1+\Big(\Theta_2+\Theta_3\upeta\Big)^2$,
which implies the proof of the Lemma.
\end{proof}

\begin{lemma}
\label{lem:weaklycontinous}
 $\mathcal{R}$ is weakly continuous from $D$ into $D$.
\end{lemma}   
 \begin{proof}
It is clear that if $(f_n,g_n)\rightharpoonup (f,g)$ weakly in $[H^1(0,T)]^2$,
then $\mathcal{R}(f_n,g_n)\rightharpoonup \mathcal{R}(f,g)$ weakly in $[H^1(0,T)]^2$
as well.
\end{proof}

 \begin{lemma}
\label{lem:contraction}
Assume that (H$_1$)-(H$_7$) are satisfied. Then there exists 
$k_o$ 
such that $\mathcal{R}^{k_o}$ is a contraction map in $L^2(0,T)$.
\end{lemma}   
\begin{proof}
Let us consider $(f_i,g_i)\in D$ and $\{u_i,w_i,\rho_i,p_i,h_i\}$,
for $i=1,2$, the corresponding solution of the direct problem.
Then, by the definition of the operator $\mathcal{R}$ we deduce
that there exists $C>0$ independent of $\{u_i,w_i,\rho_i,p_i,h_i,f_i,g_i\}$
such that
\begin{eqnarray*}
\|\mathcal{R}(f_1,g_1)-\mathcal{R}(f_2,g_2)\|_{L^2(0,t)}
\le C\Big[
\|\rho_1-\rho_2\|_{L^2(Q_t)}
+\|u_1-u_2\|_{L^2(Q_t)}
\\
+\|w_1-w_2\|_{L^2(Q_t)}
+\|\nabla(h_1-h_2)\|_{L^2(Q_t)}
+\|\nabla(r_1-r_2)\|_{L^2(Q_t)}
\Big]
\end{eqnarray*}
with $Q_t=\Omega\times [0,t]$. 
Hence, by Lemma~\ref{lema:direct_problem_solution} we 
deduce that
\begin{eqnarray*}
\|\mathcal{R}(f_1,g_1)-\mathcal{R}(f_2,g_2)\|_{L^2(0,t)}
\le \Big(C \int_0^t \|(f_1,g_1)-(f_2,g_2)\|_{L^2(0,\tau)}d\tau\Big)^{1/2}.
\end{eqnarray*}
By induction on $n$, we deduce that
\begin{eqnarray*}
\|\mathcal{R}^n(f_1,g_1)-\mathcal{R}^n(f_2,g_2)\|_{L^2(0,T)}
\le \left(\frac{C^nT^n}{n!}\right)^{1/2}
\|(f_1,g_1)-(f_2,g_2)\|_{L^2(0,T)},
\end{eqnarray*}
which implies the conclusion of the Lemma.
\end{proof}

\subsection{Proof of Theorem~\ref{teo:wellposeIP}}
\label{subsec:proofteoIP}

Thanks to Theorem~\ref{teo:ip_vs_operator_ec} we follow that 
the cornerstone to prove Theorem~\ref{teo:wellposeIP}  
is the proof of the unique
solvability of the operator equation~\eqref{eq:operador_ecuation}.
To this end, we first note that the Lemmas \ref{lem:Rmapitself}
and \ref{lem:weaklycontinous}
guarantee that the operator $\mathcal{R}$ satisfies the 
hypothesis of the following Tikhonov fixed point theorem:

\begin{theorem}
\label{teo:tikhonov}
Let $D$ be a non-empty bounded closed convex subset of a separable reflexive
Banach space $E$ and let $\mathcal{R}:D\to D$ be a weakly continuous mapping. Then
$\mathcal{R}$ has at least one fixed point in $D$.
\end{theorem} 

\noindent
Hence, there exists a solution of \eqref{eq:operador_ecuation}.
Meanwhile, the local uniqueness follows by Lemma~\ref{lem:contraction}

\section*{Acknowledgment}
We acknowledge the support of the research projects 124109 3/R 
(Universidad del B{\'\i}o-B{\'\i}o, Chile),
121909 GI/C (Universidad del B{\'\i}o-B{\'\i}o, Chile),
Fondecyt 1120260 and MTM 2012-32325 (Spain).

\appendix
\section*{Appendix}

\section{Constants $\Gamma^u_i$ and $\Gamma^w_i$ for $i\in\{1,2,3\}$}
\label{app:constantes_gamma}
We denote the right side of \eqref{lema:direct_problem_solution:ec4} by
$\mathcal{E}_1+\mathcal{E}_2+\mathcal{E}_3$ 
where $\mathcal{E}_i,\; i\in\{1,2,3\},$ are defined and bounded
as follows
\begin{eqnarray*}
&&
\mathcal{E}_1:=\|(\bu_2)_t(\cdot,t)\|_{L^s}\|\delta \rho(\cdot,t)\|_{L^2}
	\|\delta \bu(\cdot,t)\|_{L^s}
+\|\bu_1(\cdot,t)\|_{L^{\infty}}\|\nabla \bu_2(\cdot,t)\|_{L^s}
	\|\delta \rho(\cdot,t)\|_{L^2}
	\|\delta \bu(\cdot,t)\|_{L^{\overline{s}}}
\nonumber
\\
&&
\hspace{0.9cm}
+|f_1(t)|\|\nabla h_1(\cdot,t)\|_{L^s}\|\delta \rho(\cdot,t)\|_{L^2}
	\|\delta \bu(\cdot,t)\|_{L^{\overline{s}}}
+|f_1(t)|\|\bm(\cdot,t)\|_{L^\infty}
	\|\delta \rho(\cdot,t)\|_{L^2}\|\delta \bu(\cdot,t)\|_{L^2}
\nonumber
\\
&&
\hspace{0.5cm}
\le
C_{e_1}(t)\; 
\|\delta \rho(\cdot,t)\|^2_{L^2}
+\frac{3\epsilon_\bu}{2}C_{poi} \; \|\nabla \delta \bu(\cdot,t)\|^2_{L^2}
+\frac{\epsilon_\bu}{2\alpha}\; \|\sqrt{\delta_1}\delta \bu(\cdot,t)\|^2_{L^2},
\nonumber
\\
&&
\mathcal{E}_2:=
\|\rho_2(\cdot,t)\|_{L^\infty}\|\nabla \bu_2(\cdot,t)\|_{L^s}
	\|\delta \bu(\cdot,t)\|_{L^2}
	\|\delta \bu(\cdot,t)\|_{L^{\overline{s}}}
+2\mu_r \|\nabla\delta \bw(\cdot,t)\|_{L^2} \|\delta \bu(\cdot,t)\|_{L^2}
\\
&&
\hspace{0.9cm}
+|f_2(t)|\|\nabla \delta h(\cdot,t)\|_{L^2}\|\rho_2(\cdot,t)\|_{L^\infty}
	\|\delta \bu(\cdot,t)\|_{L^2},
\nonumber
\\
&&
\hspace{0.5cm}
\le
C_{e_2}(t)\;\|\sqrt{\delta_1}\delta \bu(\cdot,t)\|^2_{L^2}
+\frac{\epsilon_\bu}{2}C_{poi} \; \|\nabla \delta \bu(\cdot,t)\|^2_{L^2}
+\frac{\epsilon_h}{2} \; \|\nabla \delta h(\cdot,t)\|^2_{L^2}
+\frac{\epsilon_\bw}{2} \; \|\nabla \delta \bw(\cdot,t)\|^2_{L^2},
\\
&&
\mathcal{E}_3:=
|\delta f(t)|\|\nabla h_1(\cdot,t)\|_{L^s}\|\rho_2(\cdot,t)\|_{L^\infty}
	\|\delta \bu(\cdot,t)\|_{L^2}|\Omega|^{1/\overline{s}}
\\
&&
\hspace{0.9cm}
+|\delta f(t)|\|\bm(\cdot,t)\|_{L^\infty}
\|\rho_2(\cdot,t)\|_{L^\infty}\|\delta \bu(\cdot,t)\|_{L^2}
|\Omega|^{1/2}
\\
&&
\hspace{0.5cm}
\le
C_{e_3}(t) |\delta f(t)|^2
+\frac{\epsilon_\bu}{\alpha} \|\sqrt{\rho_1}\delta \bu(\cdot,t)\|^2_{L^2},
\end{eqnarray*}
where
\begin{eqnarray*}
C_{e_1}(t)&=&\frac{1}{2\epsilon_u}
\Big\{
\|(\bu_2)_t(\cdot,t)\|^2_{L^s}
+|f_1(t)|^2\|\nabla h_1(\cdot,t)\|^2_{L^s}
\\&&
+|f_1(t)|^2\|\bm(\cdot,t)\|^2_{L^\infty}
+\|\bu_1(\cdot,t)\|^2_{L^{\infty}}\|\nabla \bu_2(\cdot,t)\|^2_{L^s}
\Big\},
\\
C_{e_2}(t)&=&\frac{1}{\alpha}
\left\{
\frac{1}{2\epsilon_u}\|\rho_2(\cdot,t)\|^2_{L^\infty}\|\nabla \bu_2(\cdot,t)\|^2_{L^s}
+\frac{2\bu^2_r}{\epsilon_w}
+\frac{1}{\epsilon_h}|f_2(t)|^2\|\rho_2(\cdot,t)\|^2_{L^\infty}
\right\}\mbox{ and}
\\
C_{e_3}(t)&=&
\frac{1}{2\epsilon_\bu}
\Big\{\|\nabla h_1(\cdot,t)\|^2_{L^s}|\Omega|^{2/s}
  +\|\bm(\cdot,t)\|^2_{L^\infty}|\Omega|\Big\}\|\rho_2(\cdot,t)\|^2_{L^\infty}.
\end{eqnarray*}
Then, we deduce the inequality~\eqref{lema:direct_problem_solution:ec5} with
\begin{eqnarray*}
\Gamma^\bu_1(t)=\max\left\{C_{e_2}(t)+\frac{3\epsilon_\bu}{2\alpha},C_{e_1}(t)\right\},
\quad
\Gamma^\bu_2(t)=\frac{\epsilon_h}{2}
\quad
\mbox{and}
\quad
\Gamma^\bu_3(t)=C_{e_3}(t).
\end{eqnarray*}
Meanwhile, for \eqref{lema:direct_problem_solution:ec6} we proceed in a similar way
and obtain that
\begin{eqnarray*}
\Gamma^\bw_1(t)=\max\left\{C_{f_2}(t)+\frac{3\epsilon_\bw}{2\alpha},C_{f_1}(t)\right\},
\quad
\Gamma^\bw_2(t)=\frac{\epsilon_g}{2}
\quad
\mbox{and}
\quad
\Gamma^\bw_3(t)=C_{f_3}(t),
\end{eqnarray*}
with
\begin{eqnarray*}
C_{f_1}(t)&=&\frac{1}{2\epsilon_\bw}
\Big\{
\|(\bw_2)_t(\cdot,t)\|^2_{L^s}
+|g_1(t)|^2\|\nabla r_1(\cdot,t)\|^2_{L^s}
\\&&
+|g_1(t)|^2\|\bq(\cdot,t)\|^2_{L^\infty}
+\|\bw_1(\cdot,t)\|^2_{L^{\infty}}\|\nabla \bw_2(\cdot,t)\|^2_{L^s}
\Big\},
\\
C_{f_2}(t)&=&\frac{1}{\alpha}
\left\{
\frac{1}{2\epsilon_\bw}\|\rho_2(\cdot,t)\|^2_{L^\infty}\|\nabla \bw_2(\cdot,t)\|^2_{L^s}
+\frac{2\bu^2_r}{\epsilon_\bu}
+\frac{1}{\epsilon_g}|g_2(t)|^2\|\rho_2(\cdot,t)\|^2_{L^\infty}
\right\}\mbox{ and}
\\
C_{f_3}(t)&=&
\frac{1}{2\epsilon_\bw}
\Big\{\|\nabla r_1(\cdot,t)\|^2_{L^s}|\Omega|^{2/s}
  +\|\bq(\cdot,t)\|^2_{L^\infty}|\Omega|\Big\}\|\rho_2(\cdot,t)\|^2_{L^\infty}.
\end{eqnarray*}

\end{document}